\documentclass[11pt,a4paper]{article}
\usepackage{amsthm,amsmath,amssymb,graphicx,color,array}
\usepackage[a4paper, total={6.5in,8in}]{geometry}
\usepackage{extarrows,slashbox,caption}
\usepackage{multirow}
\usepackage[T1]{fontenc}
\usepackage{amssymb,amsfonts,amsmath,textcomp,amsthm}
\usepackage{longtable}
\usepackage{fullpage}
\usepackage[]{geometry}
\usepackage{changepage}
\usepackage{lscape}
\usepackage{times}
\usepackage{tikz}
\usepackage{graphicx,multirow}
\usepackage[flushleft]{threeparttable}

\theoremstyle{definition}
\theoremstyle{definition}
\newtheorem{thm}{Theorem}

\newtheorem{Def}[thm]{Definition}

\newtheorem{ex}{Example}

\begin{document}
	\title{
	A Multi-Server Retrial Queueing Inventory System With Asynchronous Multiple Vacations
}
\author{ { K. Jeganathan{$^1$}},{ T. Harikrishnan{$^2$}}, {K. Prasanna Lakshmi{$^3$}}  and { D. Nagarajan{$^{4,*}$}}\\\\		
	$\null^{1}$Ramanujan Institute for Advanced Study in Mathematics,\\
	University of Madras, Chennai, India.\\
	Email: kjeganathan@unom.ac.in	\\\\
	$\null^{2}$Department of Mathematics, Guru Nanak College(Autonomous),\\University of Madras, Chennai, India.\\
	Email:harikrishanan@gurunanakcollege.edu.in\\\\
	$\null^{3}$Department of Mathematics with Computer Apllications, \\Ethiraj College for Women, Chennai, India.\\
	Email: prasana.lakshmi99@gmail.com \\\\
	$\null^{4}$Department of Mathematics, Rajalakshmi Institute of Technology, Chennai, India.\\
	Email: dnrmsu2002@yahoo.com\\\\
	$*$ Corresponding author
}

\date{}
\newtheorem{lemma}{\bf Lemma}
\newtheorem{theorem}{\bf Theorem}
\numberwithin{equation}{section}
\maketitle
\begin{abstract}
   This article deals with asynchronous server vacation and customer retrial facility in a multi-server queueing-inventory system. The Poisson process governs the arrival of a customer. The system is comprised of $c$ identical servers, a finite-size waiting area, and a storage area containing $S$ items. The service time is distributed exponentially. If each server finds that there are an insufficient number of customers and items in the system after the busy period, they start a vacation. Once the server’s vacation is over and it recognizes there is no chance of getting busy, it goes into an idle state if the number of customers or items is not sufficient, otherwise, it will take another vacation. Furthermore, each server's vacation period occurs independently of the other servers. The system accepts a $(s, Q)$ control policy for inventory replenishment. For the steady-state analysis, the Marcel F Neuts and B Madhu Rao matrix geometric approximation approach is used owing to the structure of an infinitesimal generator matrix. The necessary stability condition and $\mathbb{R}$ matrix are to be computed and presented. After calculating the sufficient system performance measures, an expected total cost of the system is to be constructed and numerically incorporated with the parameters. Additionally, numerical analyses will be conducted to examine the waiting time of customers in the queue and in orbit, as well as the expected rate of customer loss.
	\end{abstract}
	\begin{tabbing}
		********* \= \kill
	\end{tabbing}

\section{Introduction}\label{sec1}

\indent \indent Queueing inventory models are more realistic and have more space for study compared to real-world waiting-line systems and servers that are allowed to take vacations. Giving a server time off usually keeps it operating efficiently and increases its lifespan. The continuous operation of a server in the system may cause physical or emotional stress. As a result, the vacation gives the human server a chance to unwind and recharge. Psychiatrists have also claimed that allowing employees to take breaks during their working hours can reduce their stress levels and enhance their productivity.

As per the survey by the Society for Human Resource Management, 67\% of businesses provided a bank of paid time off in 2022. The same survey found that 99\% of companies offered paid vacations, while only 20\% provided paid sick days and paid mental health days, and just 6\% offered unlimited paid time off. In 2019, the World at Work Association conducted a poll of employers as part of a study on employee benefits, including the types of paid time off provided. On average, employers offered 17 days of paid time off for workers with less than a year of service, 18 days for those with between one and two years of service, and so on.

As evidenced by the real-world component, research into the queueing-inventory model's exploration of the vacation process is a necessary area of study. Despite the fact that there are numerous research concepts in the queueing-inventory system regarding the vacation process, the vacation process in multi-server service facilities has not yet been thoroughly researched. This study will thus persuade a multi-server queueing-inventory system's vacation procedure.

\subsection{Motivation}\label{subsec1}
This mathematical model is motivated by a recent visit to a mobile phone store. It works with multiple salespeople to sell mobile phones to arriving customers. During peak hours, the customer arrival rate is higher than during non-peak hours. In such an instant, all sales executives will be busy handling the arriving customers. Due to the long service process, the sales executives feel tired and need some relaxation. Customers are allowed to take a rest period for an exponential amount of time if their arrival rate decreases. After coming back from vacation, they start their work with enthusiasm. The author used this circumstance as inspiration to develop a mathematical model for the stochastic queueing-inventory system. 

\subsection{Literature Review}\label{subsec2}
\indent \indent The existing literature review of the queueing inventory system(QIS) shows us that, in earlier days, researchers believed that customers received the service immediately after they arrived at the system. Yet, with real-life phenomena, this is not always the case. For instance, when a customer visits a bike showroom, a salesperson usually spends some time explaining the features and benefits of the product. To model such service processes accurately in a stochastic queuing-inventory system (QIS), the concept of positive service time was introduced. This concept was initially proposed by Melikov \cite{mel}, Sigman and Simchi-Levi \cite{sigsi}. Berman et al. \cite{berman} developed a deterministic service time model for a dynamic problem. Several studies have further discussed the concept of positive service time in QIS, including \cite{mag, nish, krishsan, krish15}.

A customer who attempts to reach customer care but disconnects due to a busy line and then periodically tries again until successfully connected to a server is known as a retrial customer. When the waiting room reaches its maximum capacity, no additional customers will be admitted. They can leave the system for a short period of time to take a break or perform other tasks before rejoining the line. This retry notion is crucial to the study of queueing inventory theory. The retrial concept in inventory systems was first explored by Artaljeo et al. \cite{ak}. Paul Manuel et al. \cite{palman} investigated a perishable inventory system with a retrial and negative demand. Typically, the retry of a particular client is independent of the other consumers in its orbit. Nevertheless, the rate of retrial consumers is directly proportional to orbit level; this inventory system is known as the "classical retrial inventory system(CRIS)". Ushakumari \cite{usha} developed an analytical strategy for CRIS. Nithya et. al. \cite{nssj} analyzed the two commodities in the single-server QIS with a classical retrial facility.

In the queueing-inventory model, a server is able to take a vacation if it cannot offer the subsequent service. Once the server's busy phase ends, the server in the system rests for an arbitrary period of time. Practically, allowing the server to take a rest helps it recover its energy and keep offering its services without getting tired. Observing the QIS, the possibility of a server completing his busy period is as follows: 1) When no inventory is available 2) The system does not have any customers. The QIS includes details on the rules for returning from vacations. While completing vacation, the server checks the queue and inventory level. After the vacation is over, the server checks the number of people in line and the number of items in stock. If there are signs that the load is getting heavier, the server will quickly start serving customers again. Otherwise, they may take additional vacations, referred to as multiple vacations. In addition, numerous scholars in the field of QIS have examined other types of vacation, such as delayed vacation,  working vacation, differentiated vacation, etc.

Daniel and Ramanarayanan \cite{2d} introduced the concept of vacation within a QIS. They used the phrase "rest period" to refer to a server's vacation time. In \cite{3d}, it was allowed for the server to take a break once the inventory count reached zero, and it was assumed that the duration of the break followed a general distribution. Any customers who arrive during the server's break are regarded as lost. In a marketing department, each sales executive is given a monthly target. Once one completes his target, he is allowed to take a vacation. Here, each server can complete its target individually without depending on others. Recently, Yuying Zhang \cite{yuy} contemplated multiple vacations for a QIS server. The replenishment process was implemented under  a random order size policy and the lost sales during the server's absence were also observed. Dequan Yue and Yaling Qin \cite{yue} studied production QIS, assuming that the manufacturing facility would take a vacation when the current inventory level reached $S$ items and would end the vacation as soon as the quantity of inventory items fell below $s$. Similarly, it terminates the vacation whenever the quantity of inventory items falls below $s$. Koroliuk et. al. \cite{koro} investigated the QIS server's vacation assignment. It can be noted that the server is allowed to take a break if there are no customers waiting in the queue, and they also permit the server to terminate their vacation when the queue size reaches a predetermined threshold. Otherwise, his vacation time will be extended. For more details about the QIS with vacation facility, readers can refer to the following papers \cite{rsa, ags, cr}.

Yadavalli and Jeganathan  \cite{yad} examined a perishable QIS having dual heterogeneous servers. One of the servers is capable of taking multiple breaks while the other is working. In their research on finite retrial queueing systems with server vacations, Jeganathan et.al \cite{jegreiy}) found that utilizing heterogeneous servers was the most effective strategy when compared to using homogeneous servers. Sugapriya et al. \cite{sugapr} analyzed an inventory system where the arrival of customers depends on the number of items in stock and the server takes a vacation when the stock level is empty. Sivakumar \cite{keypap} studied a retrial inventory system with multiple server vacations, where the server extends the vacation if no items are available in the inventory.

Numerous businesses in the inventory industry concentrate on enhancing their service facilities to meet the demands of their customers. This applies to a variety of establishments, such as car dealerships, clothing stores, jewellery shops, and others that use multi-server service facilities. These real-life scenarios inspired many authors to specialize in stochastic modeling, leading to an interest in increasing the number of servers in the QIS among numerous authors. Krishnamoorthy et al. \cite{Kri} utilized a multi-server retrial queue in  QIS, and proposed an algorithmic approach for it. Yadavalli et al. discussed a multi-server service facility in their finite population QIS, where a customer enters orbit if all servers are occupied. They applied the Laplace-Stieltjes transform to determine the customer's unconditional waiting time \cite{Yad}. Fong-Fan Wang et al. \cite{Fon} analyzed a priority multi-server service with level-dependent retrial QIS and used a generalized stochastic Petri net to study parameter changes and service rate heterogeneity. Further information on multi-server QIS can be found in the cited papers \cite{Ano, Chak1, Fon1, Gabi, Nai, Pau, Raj, Ras}.

Jeganathan et. al. \cite{Jeg9}  proposed two types of multi-server service facilities to enhance the multi-server QIS by selling and serving items separately. The number of users in the system varies unpredictably and fluctuates between peak and normal hours. For instance, during festivals, more customers visit textile shops than on regular days, while the number of customers is lower on regular days. Meanwhile, Harikrishnan et al. \cite{Hari} investigated the server activation policy within a queuing-inventory system that features homogeneous servers and multiple threshold stages.

Bright and Taylor \cite{Bri} analysed the calculation of equilibrium distributions in level-dependent QBD processes and presented algorithms for computing the equilibrium distributions. Zaiming Liu et.al \cite{Zaim} present a detailed analysis of an inventory system in a retrial queue with a level-dependent retrial rate, denoted as $M/PH/1$. They derive several significant findings from their analysis, such as the steady-state probability distribution of inventory levels, the steady-state probability distribution of customers in the system, and the average waiting time of customers in the system. Furthermore, the authors examine the impact of several system parameters, including the arrival rate, service rate, and retrial rate, on system performance. Reshmi and Jose \cite{Res1} perform a mathematical analysis of a perishable inventory system with dependent retrial loss. According to the authors, the system consists of a single server, an exponential service time, retrials with dependent retrial loss, and a Markovian arrival process (MAP). A $(s, S)$ inventory policy controls the level of inventory, and the demand rate is time-dependent. The study ends with a numerical example that highlights the analytical findings and shows how the dependent retry loss affects the system's performance. Rejitha and Jose \cite{Reji} propose a mathematical analysis of a stochastic inventory system with two service modes and customer retrials. The authors characterise the system as a single server with two service modes, fast and slow, with exponentially increasing retrial times. The level of inventory is governed by a $(s, S)$ inventory policy, and the demand rate varies with time. The study concludes with a numerical example that exhibits the influence of the two modalities of service on system performance and displays the analytical results.
Narayanan et.al \cite{Nara} demonstrate that the system has a single server with exponential service time, a vacation period, and a stochastic lead time with correlation. In this paper, the authors analysed the level-dependent quasi-birth-death process and examine the effects of service time, server vacation, and correlated lead time.

\subsection{Research Gap}\label{subsec3}
The stochastic queueing-inventory literature has not yet discussed the asynchronous vacation in the multi-server system. To the best of our knowledge, the idea of offering an asynchronous vacation to the servers in the stochastic inventory system is to be considered a research gap in this domain. 

\subsection{Novelty of the Model}\label{subsec4}
This paper introduces the asynchronous multiple vacation offer to the servers in a multi-server queueing-inventory system (MSQIS). In addition, the servers can do three different jobs: 1. busy; 2. idle; and 3. vacation, according to the situation that occurs in MSQIS. Further, the combination of asynchronous vacation and $(s, Q)$ ordering principle in the MSQIS is also a new attempt.

\subsection{Contribution of the Model}\label{subsec5}
The contribution of the paper is described as follows: 
\begin{enumerate}
	\item The considered queueing-inventory system admits asynchronous multiple vacations to its servers when the respective situation occurs.   
	\item The asynchronous multiple vacation policy is to be applied in the retrial queueing-inventory system. 
	\item The multi-server vacation queueing-inventory system performs well with the $(s, Q)$ continuous review and reorder setup. 
	\item This paper finds the stability condition and stationary probability vector using the Neuts and Rao matrix geometric approximation method. 
	\item This paper explains how to calculate the expected total cost and fraction time of a successful retrial rate for Orbit customers.   
	\item Further, it investigates the proposed model with some numerical elucidations.
\end{enumerate}
This paper is put together in the following order: notations and model formulation, model analysis and solutions, system performance measures, cost analysis, numerical investigations, and conclusions. 
\section{Notations and Model Formulation}\label{mod}	
\subsection{Notations}{
	$\begin{array}{lcl}
		{0^*, Q^*}&:&\text {All servers are in vacation mode when the items are $0$ and $Q$ respectively }\\
		{\bf 0}&:&\text {A matrix containing all entries are zero}\\
		{\bf e}&:&\text {A column vector of appropriate dimensions}\\
		&&\text {with each coordinate containing a value of one}\\	
		I&:&\text {An identity matrix}\\	
		\delta_{ij}&:&\left\{ \begin{array}{ll}
			1, & \mbox{if} \ j=i, \\
			0, & \mbox{otherwise}
		\end{array}\right.\\
		H(x)&:&\left\{ \begin{array}{ll}
			1, & \mbox{if} \ x\geq 0, \\
			0, & \mbox{otherwise}
		\end{array}\right.\\
		\lambda&:& \text{Arrival rate of new customer}\\
		\theta&:& \text{The rate at which a customer in orbit attempts to retry}\\
		\mu&:& \text{The rate at which a customer receives service}\\
		\eta&:& \text{Vacation completion rate of a server}\\
		\bar{\delta}_{ij}&:& 1-\delta_{ij} \\
		\lbrack D_{l}\rbrack_{x,y}&:&\left\{ \begin{array}{ll}
			p\lambda, & \mbox{if} \ x=y=l, \\
			0, & \mbox{otherwise~~~~~~~~~~~~~ where} ~D_l~ \text{is a matrix of order}~l
		\end{array}\right.\\
		\lbrack E_{l}\rbrack_{x,y}&:&\left\{ \begin{array}{ll}
			0, & \mbox{if} \ x=y=l, \\
			\iota_{1}\theta, & \mbox{otherwise~~~~~~~~~~~~~ where} ~E_l~ \text{ is a matrix of order}~l
		\end{array}\right.\\
		P_1(t)&:&\text{The number of customers in orbit at time t}\\
		P_2(t)&:&\text{The level of inventory at time t}\\
		P_3(t)&:&\text{The number of servers on vacation at time t}\\
		P_4(t)&:&\text{The number of servers busy with customers at time t.}\\
		P_5(t)&:&\text{The number of idle servers at the time of t}\\
		P_6(t)&:&\text{The number of customers in the waiting hall  at time t}\\
		ch &:& \text{Holding cost per unit item}\\
		cs &:& \text{Set up cost per order}\\
		co &:& \text{Waiting cost for a customer in the orbit per unit time}\\
		cw &:&\text{Waiting cost for a customer in the queue per unit time}\\
		cl &:&\text{Lost Cost per customer in the queue per unit time}
	\end{array}
	$
}

\begin{Def}[Asynchronous Vacation]
	In the system, each server can take or terminate vacation individually; this is called asynchronous vacation.
\end{Def}
\subsection{Model Formulation}\label{subsec7}
The considered queueing-inventory system explores the performance of a multi-server service facility with asynchronous multiple vacations.  Initially, the system has a maximum of $S$ units of an item in the inventory storage space, there is a finite waiting hall of size $N$  and a maximum of $c$ identical servers who can provide service to the customer at an instant. Upon the completion of a customer's service, one item will be deducted from the inventory storage space.

Depending on the circumstances that arise in the system, servers can perform three different tasks. The following is a list of such jobs: 1. being preoccupied with the customer to close a sale; 2. sitting around in the system; and 3. taking a break from the system. 
\begin{enumerate}
	\item[a)]{\textbf{Occurrence of Servers Busy Period:}} The customer in the queue requires that a number of servers begin to attend if there is a sufficient number of items in the storage space equal to the number of customers in the waiting hall, and vice versa.
	\item[b)]{\textbf{Occurrence of Servers Idle Period:}}
	\begin{itemize}
		\item The server goes to an idle state at the end of each service if the number of customers (except servicing customers) in the queue is insufficient for the number of items (except servicing items) in the storage space, and vice versa.
		\item The server goes to idle state at end of the each vacation period if either of the customer level or stock level is insufficient.
	\end{itemize} 
	\item[c)]{\textbf{Occurrence of Servers Vacation Period:}}
	\begin{itemize}
		\item The server goes into vacation state at the end of each service if both customer and stock levels in the system (except servicing customers and servicing items) are zero.
		\item At the end of each vacation, if the server finds that both customer and stock levels in the system (except servicing customers and servicing items) are zero, then he takes one more vacation.
	\end{itemize} 
	
\end{enumerate} 
Customers follow a Poisson distribution when arriving at the waiting hall, with the rate of arrival determined by $\lambda$. If the waiting room is at full capacity, a new customer either enters an infinite capacity orbit with probability $p$, or leaves the system permanently with probability $1-p$. This orbit is a virtual waiting area designed to minimize customer loss in the system. Under the classical retrial policy, any customer from the orbit can enter the waiting hall if the number of customers in the waiting hall is less than $N$. The retrial rate of a customer from the orbit is denoted as $\theta$, and the inter-arrival time between two successful retrial arrivals from the orbit is assumed to be exponentially distributed.\\

The service time is distributed exponentially and the service rate is defined by the parameter $\mu$. When the system's current stock level falls to a reorder level of $s$, a reorder of $Q=(S-s)$ items is triggered at that epoch. This kind of reorder process is said to be an $(s, Q)$ replenishment principle. Lead time follows an exponential distribution with the rate $\beta$. The graphical representation is illustrated in Figure \eqref{fig:r2}.
\begin{figure}[h!]

	\tikzset{every picture/.style={line width=0.75pt}}   
	\begin{tikzpicture}[x=0.75pt,y=0.75pt,yscale=-1,xscale=1]
		
		\draw  [draw opacity=0] (472,133) -- (522,133) -- (522,383) -- (472,383) -- cycle ; \draw    ; \draw   (472,183) -- (522,183)(472,233) -- (522,233)(472,283) -- (522,283) ; \draw   (472,133) -- (522,133) -- (522,383) -- (472,383) -- cycle ;
		\draw    (472,333) -- (522,333) ;
		\draw    (148,364) -- (285,364) ;
		\draw [shift={(287,364)}, rotate = 180] [color={rgb, 255:red, 0; green, 0; blue, 0 }  ][line width=0.75]    (10.93,-3.29) .. controls (6.95,-1.4) and (3.31,-0.3) .. (0,0) .. controls (3.31,0.3) and (6.95,1.4) .. (10.93,3.29)   ;
		\draw  [draw opacity=0] (211.53,279.63) -- (211.47,229.63) -- (461.47,229.37) -- (461.53,279.37) -- cycle ; \draw    ; \draw   (261.53,279.58) -- (261.47,229.58)(311.53,279.53) -- (311.47,229.53)(361.53,279.47) -- (361.47,229.47)(411.53,279.42) -- (411.47,229.42) ; \draw   (211.53,279.63) -- (211.47,229.63) -- (461.47,229.37) -- (461.53,279.37) -- cycle ;
		\draw    (411.53,279.42) -- (411.47,229.42) ;
		
		\draw   (81,189) -- (168,189) -- (168,237) -- (81,237) -- cycle ;
		\draw    (522,270) -- (605,270) ;
		\draw [shift={(607,270)}, rotate = 180] [color={rgb, 255:red, 0; green, 0; blue, 0 }  ][line width=0.75]    (10.93,-3.29) .. controls (6.95,-1.4) and (3.31,-0.3) .. (0,0) .. controls (3.31,0.3) and (6.95,1.4) .. (10.93,3.29)   ;
		\draw    (498,72.5) -- (498,112) ;
		\draw [shift={(498,114)}, rotate = 270] [color={rgb, 255:red, 0; green, 0; blue, 0 }  ][line width=0.75]    (10.93,-3.29) .. controls (6.95,-1.4) and (3.31,-0.3) .. (0,0) .. controls (3.31,0.3) and (6.95,1.4) .. (10.93,3.29)   ;
		\draw   (467.5,35.9) .. controls (467.5,30.98) and (471.48,27) .. (476.4,27) -- (566.6,27) -- (575.5,35.9) -- (575.5,71.5) -- (467.5,71.5) -- cycle ;
		\draw    (168,259) -- (209,259) ;
		\draw [shift={(211,259)}, rotate = 180] [color={rgb, 255:red, 0; green, 0; blue, 0 }  ][line width=0.75]    (10.93,-3.29) .. controls (6.95,-1.4) and (3.31,-0.3) .. (0,0) .. controls (3.31,0.3) and (6.95,1.4) .. (10.93,3.29)   ;
		\draw    (461,253) -- (471,253) ;
		\draw    (2,216) -- (78,216) ;
		\draw [shift={(80,216)}, rotate = 180] [color={rgb, 255:red, 0; green, 0; blue, 0 }  ][line width=0.75]    (10.93,-3.29) .. controls (6.95,-1.4) and (3.31,-0.3) .. (0,0) .. controls (3.31,0.3) and (6.95,1.4) .. (10.93,3.29)   ;
		\draw   (81,237) -- (168,237) -- (168,277) -- (81,277) -- cycle ;
		\draw    (125.5,238) -- (125.5,278) ;
		\draw   (122.12,454.71) .. controls (122.12,454.71) and (122.12,454.71) .. (122.12,454.71) .. controls (144,474.71) and (143.44,507.5) .. (120.87,527.96) .. controls (98.3,548.42) and (62.26,548.8) .. (40.38,528.8) .. controls (18.5,508.81) and (19.06,476.02) .. (41.64,455.56) .. controls (68.88,430.86) and (82.5,418.51) .. (82.5,418.51) .. controls (82.5,418.51) and (95.71,430.58) .. (122.12,454.71) -- cycle ;
		\draw   (56,378) -- (76.4,310) -- (131.6,310) -- (152,378) -- cycle ;
		\draw    (62,352) -- (145,352) ;
		\draw    (103.5,352) -- (104,377) ;
		\draw    (82,378) -- (82,416) ;
		\draw [shift={(82,418)}, rotate = 270] [color={rgb, 255:red, 0; green, 0; blue, 0 }  ][line width=0.75]    (10.93,-3.29) .. controls (6.95,-1.4) and (3.31,-0.3) .. (0,0) .. controls (3.31,0.3) and (6.95,1.4) .. (10.93,3.29)   ;
		
		\draw    (439.89,539.9) -- (548.34,540.83) ;
		\draw   (494.47,500.97) -- (548.34,540.26) -- (527.09,604.03) -- (460.08,604.15) -- (439.92,540.45) -- cycle ;
		\draw    (494.45,500.97) -- (494.45,540.97) ;
		\draw   (227.55,83.28) .. controls (225.41,78.41) and (224.08,73.07) .. (223.72,67.44) .. controls (222.08,41.56) and (241.73,19.1) .. (267.62,17.29) .. controls (293.51,15.47) and (315.83,34.98) .. (317.47,60.87) .. controls (317.98,68.91) and (316.44,76.62) .. (313.29,83.48) -- cycle ;
		\draw    (308,35) -- (468,35) ;
		\draw    (318,65) -- (468.5,65) ;
		\draw   (373,59) -- (383.5,65.5) -- (373,72) ;
		\draw   (379.23,42.21) -- (369,35.29) -- (379.76,29.22) ;
		\draw    (43.71,452.68) -- (43.5,220) ;
		\draw [shift={(43.5,218)}, rotate = 89.95] [color={rgb, 255:red, 0; green, 0; blue, 0 }  ][line width=0.75]    (10.93,-3.29) .. controls (6.95,-1.4) and (3.31,-0.3) .. (0,0) .. controls (3.31,0.3) and (6.95,1.4) .. (10.93,3.29)   ;
		\draw   (173.5,498.8) .. controls (181.12,498.8) and (187.3,492.62) .. (187.3,485) -- (294.7,485) .. controls (294.7,492.62) and (300.88,498.8) .. (308.5,498.8) -- (308.5,540.2) .. controls (300.88,540.2) and (294.7,546.38) .. (294.7,554) -- (187.3,554) .. controls (187.3,546.38) and (181.12,540.2) .. (173.5,540.2) -- cycle ;
		\draw    (296.7,547) -- (433.7,547) ;
		\draw [shift={(435.7,547)}, rotate = 180] [color={rgb, 255:red, 0; green, 0; blue, 0 }  ][line width=0.75]    (10.93,-3.29) .. controls (6.95,-1.4) and (3.31,-0.3) .. (0,0) .. controls (3.31,0.3) and (6.95,1.4) .. (10.93,3.29)   ;
		\draw    (464,523) -- (312,523) ;
		\draw [shift={(310,523)}, rotate = 360] [color={rgb, 255:red, 0; green, 0; blue, 0 }  ][line width=0.75]    (10.93,-3.29) .. controls (6.95,-1.4) and (3.31,-0.3) .. (0,0) .. controls (3.31,0.3) and (6.95,1.4) .. (10.93,3.29)   ;
		
		\draw    (98,278) -- (98,306) ;
		\draw [shift={(98,308)}, rotate = 270] [color={rgb, 255:red, 0; green, 0; blue, 0 }  ][line width=0.75]    (10.93,-3.29) .. controls (6.95,-1.4) and (3.31,-0.3) .. (0,0) .. controls (3.31,0.3) and (6.95,1.4) .. (10.93,3.29)   ;
		\draw    (484,384) -- (484,502) ;
		\draw [shift={(484,504)}, rotate = 270] [color={rgb, 255:red, 0; green, 0; blue, 0 }  ][line width=0.75]    (10.93,-3.29) .. controls (6.95,-1.4) and (3.31,-0.3) .. (0,0) .. controls (3.31,0.3) and (6.95,1.4) .. (10.93,3.29)   ;
		\draw    (506,509) -- (506,384) ;
		\draw [shift={(506,382)}, rotate = 90] [color={rgb, 255:red, 0; green, 0; blue, 0 }  ][line width=0.75]    (10.93,-3.29) .. controls (6.95,-1.4) and (3.31,-0.3) .. (0,0) .. controls (3.31,0.3) and (6.95,1.4) .. (10.93,3.29)   ;

		\draw (490,305) node [anchor=north west][inner sep=0.75pt]   [align=left] {{\huge .}};
		\draw (490,319) node [anchor=north west][inner sep=0.75pt]   [align=left] {{\huge .}};
		\draw (490,294) node [anchor=north west][inner sep=0.75pt]   [align=left] {{\huge .}};
		\draw (283.51,259.54) node [anchor=north west][inner sep=0.75pt]  [rotate=-269.9] [align=left] {{\huge .}};
		\draw (297.51,259.51) node [anchor=north west][inner sep=0.75pt]  [rotate=-269.9] [align=left] {{\huge .}};
		\draw (272.51,259.56) node [anchor=north west][inner sep=0.75pt]  [rotate=-269.9] [align=left] {{\huge .}};
		\draw (500.51,51.54) node [anchor=north west][inner sep=0.75pt]  [rotate=-269.9] [align=left] {{\huge .}};
		\draw (514.51,51.51) node [anchor=north west][inner sep=0.75pt]  [rotate=-269.9] [align=left] {{\huge .}};
		\draw (489.51,51.56) node [anchor=north west][inner sep=0.75pt]  [rotate=-269.9] [align=left] {{\huge .}};
		\draw (493,5) node [anchor=north west][inner sep=0.75pt]   [align=left] {Inventory};
		\draw (304,202) node [anchor=north west][inner sep=0.75pt]   [align=left] {Waiting Hall};
		\draw (34,192.4) node [anchor=north west][inner sep=0.75pt]    {$\lambda $};
		\draw (14,319.4) node [anchor=north west][inner sep=0.75pt]    {$\iota _{1} \theta $};
		\draw (471,115) node [anchor=north west][inner sep=0.75pt]   [align=left] {Servers};
		\draw (492,150.4) node [anchor=north west][inner sep=0.75pt]    {$1$};
		\draw (491,199.4) node [anchor=north west][inner sep=0.75pt]    {$2$};
		\draw (543,248.4) node [anchor=north west][inner sep=0.75pt]    {$\iota _{3} \mu $};
		\draw (87,384.4) node [anchor=north west][inner sep=0.75pt]    {$p\lambda $};
		\draw (186,339.4) node [anchor=north west][inner sep=0.75pt]    {$( 1-p) \lambda $};
		\draw (213.5,381.25) node   [align=left] {\begin{minipage}[lt]{72.76pt}\setlength\topsep{0pt}
				Customer loss
		\end{minipage}};
		\draw (490,347.4) node [anchor=north west][inner sep=0.75pt]  [font=\large]  {$c$};
		\draw (491,249.4) node [anchor=north west][inner sep=0.75pt]    {$3$};
		\draw (82.28,193.99) node [anchor=north west][inner sep=0.75pt]   [align=left] {Waiting Hall\\ \ \ \ \ is full};
		\draw (90,249) node [anchor=north west][inner sep=0.75pt]   [align=left] {Yes};
		\draw (134,250) node [anchor=north west][inner sep=0.75pt]   [align=left] {No};
		\draw (44,477) node [anchor=north west][inner sep=0.75pt]   [align=left] {Infinite orbit};
		\draw (78.4,313) node [anchor=north west][inner sep=0.75pt]   [align=left] {Decide \\to orbit};
		\draw (69,356) node [anchor=north west][inner sep=0.75pt]   [align=left] {Yes};
		\draw (116,354) node [anchor=north west][inner sep=0.75pt]   [align=left] {No};
		\draw (379,244.4) node [anchor=north west][inner sep=0.75pt]    {$1$};
		\draw (429,244.4) node [anchor=north west][inner sep=0.75pt]    {$0$};
		\draw (327,243.4) node [anchor=north west][inner sep=0.75pt]    {$2$};
		\draw (227,244.4) node [anchor=north west][inner sep=0.75pt]    {$N$};
		\draw (456,544) node [anchor=north west][inner sep=0.75pt]   [align=left] { Possibility\\ \ to take a\\ \ Vacation};
		\draw (466,519) node [anchor=north west][inner sep=0.75pt]   [align=left] {Yes};
		\draw (496.45,519.97) node [anchor=north west][inner sep=0.75pt]   [align=left] {No};
		\draw (516.35,513.91) node [anchor=north west][inner sep=0.75pt]  [rotate=-270.6] [align=left] {Stay in the system};
		\draw (346,503) node [anchor=north west][inner sep=0.75pt]   [align=left] {Go for vacation};
		\draw (296.69,550.01) node [anchor=north west][inner sep=0.75pt]  [rotate=-0.19] [align=left] {Vacation completion};
		\draw (312,6.4) node [anchor=north west][inner sep=0.75pt]    {$( s,Q)$};
		\draw (355,7) node [anchor=north west][inner sep=0.75pt]   [align=left] {ordering policy};
		\draw (243,35) node [anchor=north west][inner sep=0.75pt]   [align=left] {Supplier\\Station};
		\draw (385,45.4) node [anchor=north west][inner sep=0.75pt]    {$\beta $};
		\draw (182.5,501.8) node [anchor=north west][inner sep=0.75pt]   [align=left] {Number of servers\\ is on vacation};
		\draw (358,529.4) node [anchor=north west][inner sep=0.75pt]    {$\eta $};
		\draw (553,39.4) node [anchor=north west][inner sep=0.75pt]    {$0$};
		\draw (533,39.4) node [anchor=north west][inner sep=0.75pt]    {$1$};
		\draw (469.5,39.3) node [anchor=north west][inner sep=0.75pt]    {$S$};

	\end{tikzpicture}
	\caption{\textbf{Graphical depiction of the model.}\label{fig:r2}}
\end{figure}
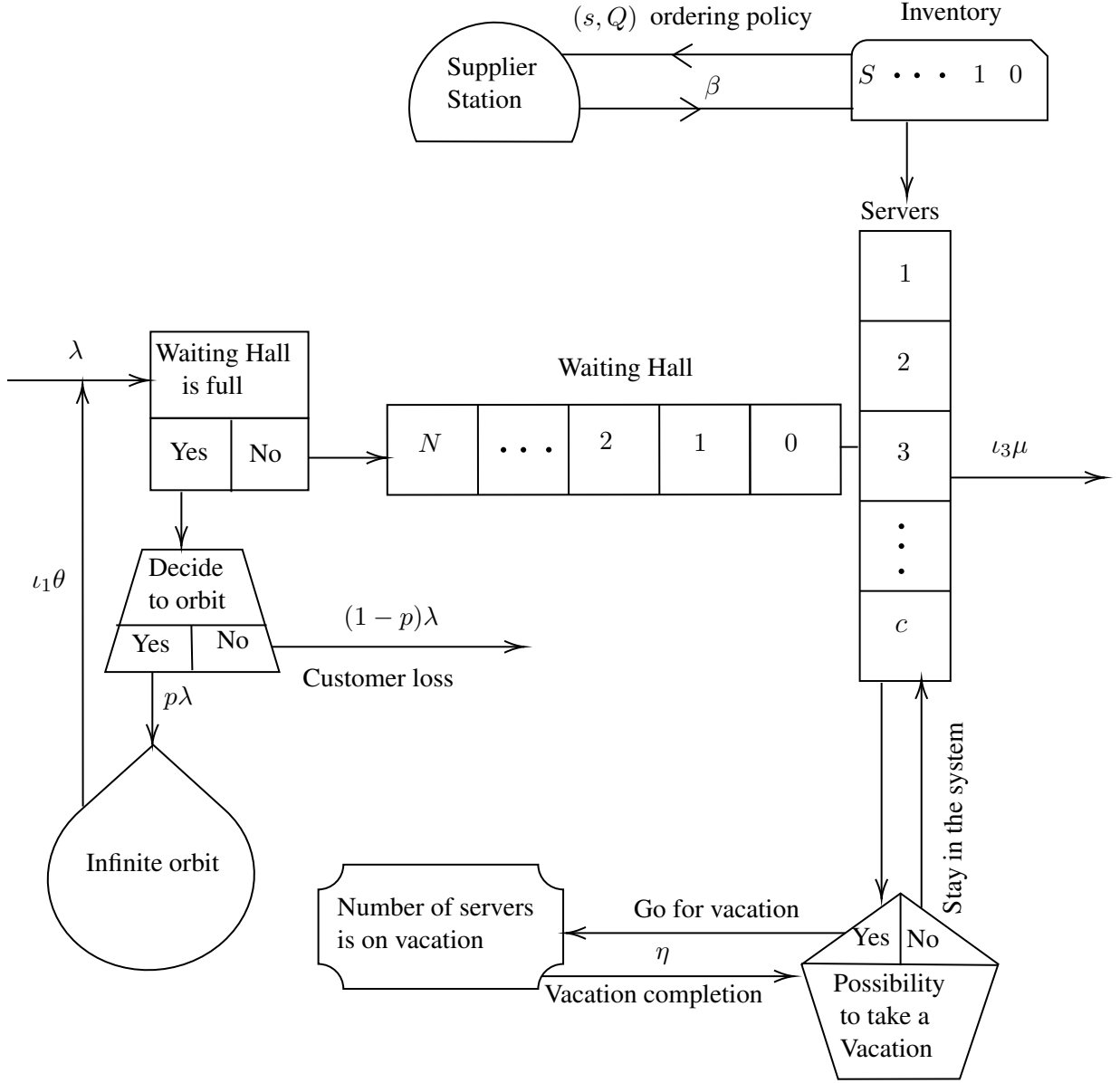
\section{Analysis and Solutions of the Model}\label{mai}
Consider a six-dimensional Markov process $$\{P(t)=(P_1(t), P_2(t),P_3(t),P_4(t),P_5(t),P_6(t)),~ t \geq 0\}$$ with state space	
$	\Omega=\bigcup\limits_{i=1}^{8}E_i$
\begin{adjustwidth}{-1cm}{1cm}
	$	E_1=\left\{(0\leq \iota_{1}<\infty,~0,~c-x_1,~0,~x_{1},~x_1 \leq \iota_{6}\leq N) ~| ~0\leq x_{1}\leq c \right\}\\
	E_2=\left\{(0\leq \iota_{1}<\infty,~x_{2}+1,~c-(j+1),~0\leq \iota_{4}\leq j,~j+1-\iota_{4},~\iota_{4})  ~|~ 1\leq x_{2}\leq c-2;0\leq j\leq x_2-1\right \}\\
	E_3=\left\{(0\leq \iota_{1}<\infty,~x_{2}+1,~c-(j+1),~j+1,~0,~j+1\leq \iota_{6}\leq N) ~|~ 1\leq x_{2}\leq c-2;0\leq j\leq x_2-1 \right\}\\
	E_4=\left\{(0\leq \iota_{1}<\infty,~j,~c-x_3,~0\leq \iota_{4}\leq x_3-1,~x_3-\iota_{4},~\iota_{4})~|~1\leq x_{3}\leq c-1;c\leq j\leq S \right\}\\
	E_5=\left\{(0\leq \iota_{1}<\infty,~j,~c-x_3,~x_3,~0,~x_3\leq  \iota_{6} \leq N )~|~1\leq x_{3}\leq c-1;c\leq j\leq S \right\}\\
	E_6=\left\{(0\leq \iota_{1}<\infty,~Q,~c,~0,~0,~0\leq  \iota_{6} \leq N) \right\}\\
	E_7=\left\{(0\leq \iota_{1}<\infty,~x_4+1,~c-(x_4+1),~0\leq \iota_{4}\leq x_4,~x_4+1-\iota_{4},~\iota_{4})~|~ 0\leq x_{4}\leq c-1\right\}\\
	E_8=\left\{(0\leq \iota_{1}<\infty,~x_5,~c-x_5,~x_5,~0,~x_5\leq\iota_{6}\leq N)~|~ 1\leq x_{5}\leq c\right\}\\
	E_9=\left\{(0\leq \iota_{1}<\infty,~x_6,~c-(x_6+j),~x_6,~j,~(x_6+j)\leq  \iota_{6} \leq N ) ~|~ 1\leq x_{6}\leq c-1 ;1\leq j\leq c-x_6\right\}\\
	E_{10}=\left\{(0\leq \iota_{1}<\infty,~x_7,~0,~0\leq \iota_{4}\leq c-1,~c-\iota_{4},~\iota_{4})~|~  c+1\leq x_{6}\leq S\right\}\\
	E_{11}=\left\{(0\leq \iota_{1}<\infty,~x_7,~0,~c,~0,~c\leq  \iota_{6} \leq N)~|~ c+1\leq x_{7}\leq S\right\}$
\end{adjustwidth}

\subsection{Construction of Infinitesimal Generator Matrix}
The process  $\{P(t), t \geq 0\}$  is called a $Q B D$ process and its infinitesimal generator matrix is given by

\begin{eqnarray}\label{n1}
	H & = & \left(
	\begin{array}{cccccccccccccccccccccc}
		{\mathbb{H}_{01}} & \mathbb{H}_{0} & \mathbf{0} & \mathbf{0} & \mathbf{0} &\cdots \\
		\mathbb{H}_{10} & \mathbb{H}_{11} & \mathbb{H}_{0} & \mathbf{0} &  \mathbf{0} & \cdots \\
		\mathbf{0} & \mathbb{H}_{20} & \mathbb{H}_{21} & \mathbb{H}_{0} & \mathbf{0} & \cdots\\
		\vdots & \vdots & \vdots & \vdots & \vdots& \ddots \\
	\end{array}\right),
\end{eqnarray}
where all the matrices in the above infinitesimal generator matrix $H$ have dimension $$Sc(N+1)+(c+2)N-(2c^3+3c^2-5c-12)(1/6)$$ and all the sub-matrices are defined as follows:

\begin{enumerate}
	\item The non-zero upper diagonal matrix, $\mathbb{H}_{0}$ describes the transitions of the new arrival to the orbit if the waiting hall is full. 
	\begin{eqnarray}\label{z1}
		\mathbb{H}_{0} & = & diag(\Bbbk_{0^*},\Bbbk_{Q^*},\Bbbk_{0},\Bbbk_{1},\cdots,\Bbbk_{c-1},\Bbbk_{c},\Bbbk_{c},\cdots,\Bbbk_{c})\end{eqnarray}
	where 
	\begin{eqnarray}\label{z2}
		\Bbbk_{j} = diag(D_{N-(c-1)},D_{N-(c-2)},\cdots,D_{N-j},D_{N+1},D_{N+1},\cdots,D_{N+1}),~j=0,1,\cdots,c
	\end{eqnarray}
	is a sub-block matrix of order $c$.
	\item The non-zero lower diagonal matrix, $\mathbb{H}_{\iota_10}$ denotes the transitions of orbit customers enter the waiting hall.
	\begin{eqnarray}\label{z3}
		\text{For}  ~~ \iota_1=1, 2, \cdots\nonumber\\
		\mathbb{H}_{\iota_10} & = & diag(\mathbb{W}_{0^*},\mathbb{W}_{Q^*},\mathbb{W}_{0},\mathbb{W}_{1},\cdots,\mathbb{W}_{c-1},\mathbb{W}_{c},\mathbb{W}_{c},\cdots,\mathbb{W}_{c})
	\end{eqnarray}
	where 
	\begin{eqnarray}\label{z4}
		\mathbb{W}_{j} = diag(E_{N-(c-1)},E_{N-(c-2)},\cdots,E_{N-j},E_{N+1},E_{N+1},\cdots,E_{N+1}),~j=0,1,\cdots,c
	\end{eqnarray}
	is a sub-block matrix of order $c$.
	\item The diagonal matrix, $\mathbb{H}_{\iota_11}$ describes the transitions of primary arrival, service process, reorder process, server taking a vacation and vacation completion process.
	\begin{eqnarray}
		\text{For}  ~~ \iota_1=0, 1, 2, \cdots\nonumber\\
		\left[\mathbb{H}_{\iota_{1}1}\right]_{\iota_{2},\iota_{2}^{'}} & = & \left\{ \begin{array}{llll}
			\mathbb{J}_{\iota_{2}^{'}} &  \iota_{2}^{'}=0^*& \iota_{2}=1\\
			\mathbb{J}_{\iota_{2}^{'}}  &  \iota_{2}^{'}=Q^*,0& \iota_{2}=0^*\\
			\mathbb{J}_{\iota_{2}^{'}}  &  \iota_{2}^{'}=Q& \iota_{2}=Q^*\\	
			\mathbb{F}_{\iota_{2}}& \iota_{2}^{'}=\iota_{2}-1 & \iota_{2}\in 1,2,\cdots,c\\
			\mathbb{F}_{c+1} & \iota_{2}^{'}=\iota_{2}-1 & \iota_{2}\in c+1,c+2,\cdots,S\\	
			\mathbb{L}_{\iota_{2}}& \iota_{2}^{'}=\iota_{2}+Q & \iota_{2}\in 0,1,2,\cdots,c\\
			\mathbb{L}_{c}& \iota_{2}^{'}=\iota_{2}+Q & \iota_{2}\in c+1,c+2,\cdots,s\\
			\mathbb{G}_{\iota_{2}^{'}} & \iota_{2}^{'}=\iota_{2} & \iota_{2}\in 0^*,Q^*,0,1,\cdots,c-1\\
			\mathbb{G}_c & \iota_{2}^{'}=\iota_{2} & \iota_{2}\in c,c+1,\cdots,S\\
			0, &\mbox{otherwise}
		\end{array}\right.\\*[0.15cm]\nonumber
	\end{eqnarray}
	where
	\begin{eqnarray}\label{z5}
		\mathbb{J}_{0^*} & = & \left\{ \begin{array}{llllll}
			\iota_4 \mu	&\iota_3=c-1 & \iota_4=1&\iota_5=0&\iota_6=1,2,\cdots,N\\
			&\iota_{3}^{'}=c & \iota_{4}^{'}=0&\iota_{5}^{'}=\iota_5&\iota_{6}^{'}=\iota_{6}-1 \\
			0, & \mbox{otherwise,}
		\end{array}\right.\\*[0.25cm]\nonumber
	\end{eqnarray}
	\begin{eqnarray}\label{z6}
		\mathbb{J}_{Q^*} & = & \left\{ \begin{array}{llllll}
			\beta	&\iota_3=c & \iota_4=0&\iota_5=0&\iota_6=0,1,2,\cdots,N\\
			&\iota_{3}^{'}=c & \iota_{4}^{'}=\iota_4&\iota_{5}^{'}=\iota_5&\iota_{6}^{'}=\iota_{6}\\
			0, & \mbox{otherwise,}
		\end{array}\right.\\*[0.25cm]\nonumber
	\end{eqnarray}
	\begin{eqnarray}\label{z7}
		\mathbb{J}_{0} & = & \left\{ \begin{array}{llllll}
			\eta	&\iota_3=c & \iota_4=0&\iota_5=0&\iota_6=1,2,\cdots,N\\
			&\iota_{3}^{'}=c-1 & \iota_{4}^{'}=\iota_4&\iota_{5}^{'}=1&\iota_{6}^{'}=\iota_{6}\\
			0, & \mbox{otherwise,}
		\end{array}\right.\\*[0.25cm]\nonumber
	\end{eqnarray}
	\begin{eqnarray}\label{z8}
		\mathbb{J}_{Q} & = & \left\{ \begin{array}{llllll}
			\eta&\iota_3=c & \iota_4=0&\iota_5=0&\iota_6=0,1,2,\cdots,N\\
			&\iota_{3}^{'}=c-1 & \iota_{4}^{'}=1-\iota_{5}^{'}&\iota_{5}^{'}=\delta_{\iota_{6}^{'}0}&\iota_{6}^{'}=\iota_{6}\\
			0, & \mbox{otherwise,}
		\end{array}\right.\\*[0.25cm]\nonumber
	\end{eqnarray}
	For $\iota_{2}\in 1,2,\cdots,c-1$
	\begin{adjustwidth}{-3cm}{1cm}
		\begin{eqnarray}\label{z9}\nonumber
			\mathbb{F}_{\iota_{2} }& = & \left\{ \begin{array}{lllllll}
				\iota_{4}\mu	&\iota_3=0,\cdots,c-(\iota_{2}+1) & \iota_4=\iota_{2}&\iota_5=c-(\iota_{3}+	\iota_{4})&\iota_6=	\iota_{4}+	\iota_{5}\\
				&\iota_{3}^{'}=\iota_3+1&\iota_{4}^{'}=\iota_{4}-1&\iota_{5}^{'}=\iota_{5}&\iota_{6}^{'}=\iota_{6}-1\\
				\iota_{4}\mu	&\iota_3=0,\cdots,c-(\iota_{2}+1) & \iota_4=\iota_{2}&\iota_5=c-(\iota_{3}+	\iota_{4})&\iota_6=	\iota_{4}+	\iota_{5}+1,\cdots,N\\
				&\iota_{3}^{'}=\iota_3&\iota_{4}^{'}=\iota_{4}-1&\iota_{5}^{'}=\iota_{5}+1&\iota_{6}^{'}=\iota_{6}-1\\	
				\bar{\delta}_{\iota_{2}1}\iota_{4}\mu&\iota_3=c-\iota_{2} ,\cdots,c-2 & \iota_4=1,\cdots,c-\iota_{3}&\iota_5=c-(\iota_{3}+\iota_{4})&\iota_6=	\iota_{4}\\
				&\iota_{3}^{'}=\iota_3+1&\iota_{4}^{'}=\iota_{4}-1&\iota_{5}^{'}=\iota_{5}&\iota_{6}^{'}=\iota_{6}-1\\
				\bar{\delta}_{\iota_{2}1}\iota_{4}\mu&\iota_3=c-\iota_{2} ,\cdots,c-2 & \iota_4=\iota_{2}&\iota_5=0&\iota_6=c-\iota_{3},\cdots,N\\
				&\iota_{3}^{'}=\iota_3&\iota_{4}^{'}=\iota_{4}-1&\iota_{5}^{'}=\iota_{5}+1&\iota_{6}^{'}=\iota_{6}-1\\	
				\iota_{4}\mu&\iota_3=c-1 & \iota_4=1&\iota_5=0&\iota_6=1,2,\cdots,N\\
				&\iota_{3}^{'}=\iota_3&\iota_{4}^{'}=\delta_{\iota_{6}^{'}2}&\iota_{5}^{'}=\delta_{\iota_{6}^{'}1}&\iota_{6}^{'}=\iota_{6}-1\\
				0, & \mbox{otherwise,}
			\end{array}\right.\\*[0.25cm]
		\end{eqnarray}
		\begin{eqnarray}\label{z10}\nonumber
			\mathbb{F}_{c }& = & \left\{ \begin{array}{lllllll}
				\iota_{4}\mu	&\iota_3=0& \iota_4=1,\cdots,c&\iota_5=c-\iota_{4}&\iota_6=	\iota_{4}\\
				&\iota_{3}^{'}=1&\iota_{4}^{'}=\iota_{4}-1&\iota_{5}^{'}=\iota_{5}-1&\iota_{6}^{'}=\iota_{6}-1\\
				\iota_{4}\mu	&\iota_3=0& \iota_4=c&\iota_5=0&\iota_6=c+1,\cdots,N\\
				&\iota_{3}^{'}=\iota_3&\iota_{4}^{'}=\iota_{4}-1&\iota_{5}^{'}=1&\iota_{6}^{'}=\iota_{6}-1\\	
				\iota_{4}\mu&\iota_3=1,\cdots,c-1 & \iota_4=1,\cdots,c-\iota_{3}&\iota_5=c-(\iota_{3}+\iota_{4})&\iota_6=\iota_{4}\\
				&\iota_{3}^{'}=\iota_3&\iota_{4}^{'}=\iota_{4}-1&\iota_{5}^{'}=\iota_{5}+1&\iota_{6}^{'}=\iota_{6}-1\\	
				\iota_{4}\mu&\iota_3=1,\cdots,c-1 & \iota_4=c-\iota_{3}&\iota_5=0&\iota_6=\iota_{4}+1,\cdots,N\\
				&\iota_{3}^{'}=\iota_3&\iota_{4}^{'}=\iota_{4}&\iota_{5}^{'}=\iota_{5}&\iota_{6}^{'}=\iota_{6}-1\\
				0, & \mbox{otherwise,}
			\end{array}\right.\\*[0.25cm]
		\end{eqnarray}
		For $\iota_{2}\in c+1,c+2,\cdots,S$
		\begin{eqnarray}\label{z11}\nonumber
			\mathbb{F}_{c+1 }& = & \left\{ \begin{array}{lllllll}
				\iota_{4}\mu&\iota_3=0,1,\cdots,c-1 & \iota_4=1,\cdots,c-\iota_{3}&\iota_5=c-(\iota_{3}+\iota_{4})&\iota_6=\iota_{4}\\
				&\iota_{3}^{'}=\iota_3&\iota_{4}^{'}=\iota_{4}-1&\iota_{5}^{'}=\iota_{5}+1&\iota_{6}^{'}=\iota_{6}-1\\	
				\iota_{4}\mu&\iota_3=0,1,\cdots,c-1 & \iota_4=c-\iota_{3}&\iota_5=0&\iota_6=\iota_{4}+1,\cdots,N\\
				&\iota_{3}^{'}=\iota_3&\iota_{4}^{'}=\iota_{4}&\iota_{5}^{'}=\iota_{5}&\iota_{6}^{'}=\iota_{6}-1\\
				0, & \mbox{otherwise,}
			\end{array}\right.\\*[0.25cm]
		\end{eqnarray}
		For $\iota_{2}\in 0,1,\cdots,c-1$
		\begin{eqnarray}\label{z12}\nonumber
			\mathbb{L}_{\iota_{2}}& = & \left\{ \begin{array}{lllllll}
				\beta	&\iota_3=0,\cdots,c-(\iota_{2}+1)& \iota_4=\iota_2&\iota_5=c-(\iota_{3}+\iota_{4})&\iota_6=\iota_{4}+	\iota_{5},\cdots,N\\
				&\iota_{3}^{'}=\iota_{3}&\iota_{4}^{'}=c-\iota_{3}^{'}&\iota_{5}^{'}=0&\iota_{6}^{'}=\iota_{6}\\
				\bar{\delta}_{\iota_{2}0}\beta	&\iota_3=c-\iota_{2},\cdots,c-1& \iota_4=0,\cdots,c-\iota_{3}&\iota_5=c-(\iota_{3}+\iota_{4})&\iota_6=\iota_{4}\\
				&\iota_{3}^{'}=\iota_3&\iota_{4}^{'}=\iota_{4}&\iota_{5}^{'}=\iota_{5}&\iota_{6}^{'}=\iota_{6}\\	
				\bar{\delta}_{\iota_{2}0}\beta&\iota_3=c-\iota_{2},\cdots,c-1 & \iota_4=\iota_{2}&\iota_5=0&\iota_6=c-\iota_{3}+1,\cdots,N\\
				&\iota_{3}^{'}=\iota_3&\iota_{4}^{'}=\iota_{4}&\iota_{5}^{'}=\iota_{5}&\iota_{6}^{'}=\iota_{6}\\	
				0, & \mbox{otherwise,}
			\end{array}\right.\\*[0.25cm]
		\end{eqnarray}
		\begin{eqnarray}\label{z13}\nonumber
			\mathbb{L}_{c}& = & \left\{ \begin{array}{lllllll}
				\beta	&\iota_3=0,\cdots,c-1& \iota_4=0,\cdots,c-\iota_{3}&\iota_5=c-(\iota_{3}+\iota_{4})&\iota_6=\iota_{4}\\
				&\iota_{3}^{'}=\iota_3&\iota_{4}^{'}=\iota_{4}&\iota_{5}^{'}=\iota_{5}&\iota_{6}^{'}=\iota_{6}\\	
				\beta&\iota_3=0,\cdots,c-1 & \iota_4=c-\iota_{3}&\iota_5=0&\iota_6=\iota_{4}+1,\cdots,N\\
				&\iota_{3}^{'}=\iota_3&\iota_{4}^{'}=\iota_{4}&\iota_{5}^{'}=\iota_{5}&\iota_{6}^{'}=\iota_{6}\\	
				0, & \mbox{otherwise,}
			\end{array}\right.\\*[0.25cm]
		\end{eqnarray}
		For $\iota_{2}\in 0^*,Q^*$
		\begin{eqnarray}\label{z14}\nonumber
			\mathbb{G}_{\iota_{2}} & = & \left\{ \begin{array}{llllll}
				-(\iota_{1}\theta+\bar{\delta}_{\iota_{6}N}\lambda+\delta_{\iota_{2}0^*}~\beta+\delta_{\iota_{2}0^*}\bar{\delta}_{\iota_{6}N}~\eta &&&&\\
				+\delta_{\iota_{2}Q^*}~\eta+\delta_{\iota_{6}N}~p\lambda)	&\iota_3=c & \iota_4=0&\iota_5=0&\iota_6=0,1,\cdots,N\\
				&\iota_{3}^{'}=c & \iota_{4}^{'}=\iota_4&\iota_{5}^{'}=\iota_5&\iota_{6}^{'}=\iota_{6} \\
				\lambda	&\iota_3=c & \iota_4=0&\iota_5=0&\iota_6=0,1,\cdots,N-1\\
				&\iota_{3}^{'}=c & \iota_{4}^{'}=\iota_4&\iota_{5}^{'}=\iota_5&\iota_{6}^{'}=\iota_{6}+1 \\
				0, & \mbox{otherwise,}
			\end{array}\right.\\*[0.25cm]
		\end{eqnarray}
		For $\iota_{2}\in 0,1,\cdots,c-1$
		\begin{adjustwidth}{-.3cm}{-1cm}
			\begin{eqnarray}\label{z15}\nonumber
				\mathbb{G}_{\iota_{2} } =  \left\{ \begin{array}{lllllll}
					\lambda	&0\leq \iota_3\leq c-(\iota_{2}+1) & \iota_4=\iota_{2}&\iota_5=c-(\iota_{3}+	\iota_{4})&\iota_{4}+	\iota_{5}\leq \iota_6\leq N-1\\
					&\iota_{3}^{'}=\iota_3&\iota_{4}^{'}=\iota_{4}&\iota_{5}^{'}=\iota_{5}&\iota_{6}^{'}=\iota_{6}+1\\
					\bar{\delta}_{\iota_{2}0}\lambda	&c-\iota_{2}\leq \iota_3\leq c-1 & 0\leq \iota_4\leq c-\iota_{3}&\iota_5=c-(\iota_{3}+	\iota_{4})&\iota_6=	\iota_{4}\\
					&\iota_{3}^{'}=\iota_3&\iota_{4}^{'}=\iota_{4}&\iota_{5}^{'}=\iota_{5}&\iota_{6}^{'}=\iota_{6}+1\\	
					\bar{\delta}_{\iota_{2}0}\lambda  &c-\iota_{2}\leq \iota_3\leq c-1 & \iota_4=\iota_{2}&\iota_5=0&c-\iota_{3}+1\leq \iota_6\leq N-1\\
					&\iota_{3}^{'}=\iota_3&\iota_{4}^{'}=\iota_{4}&\iota_{5}^{'}=\iota_{5}&\iota_{6}^{'}=\iota_{6}+1\\
					-(\iota_{4}\mu&&&&\\
					+\bar{\delta}_{\iota_{6}N}(\lambda+\iota_{1}\theta)&&&&\\
					+\bar{\delta}_{\iota_{3}0}\bar{\delta}_{\iota_{6}\iota_{5}}\eta &&&&\\+\delta_{\iota_{6}N}~p\lambda)   &0\leq \iota_3\leq c-(\iota_{2}+1) & \iota_4=\iota_{2}&\iota_5=c-(\iota_{3}+	\iota_{4})&\iota_{4}+\iota_{5}\leq \iota_6\leq N\\
					&\iota_{3}^{'}=\iota_3&\iota_{4}^{'}=\iota_{4}&\iota_{5}^{'}=\iota_{5}&\iota_{6}^{'}=\iota_{6}\\			
					-\bar{\delta}_{\iota_{2}0}(\iota_{4}\mu+&&&&\\
					(\lambda+\iota_{1}\theta)+&&&&\\
					\bar{\delta}_{\iota_{2}(c-1)}\bar{\delta}_{\iota_{3}(c-1)}\eta &&&&\\+\delta_{\iota_{6}N}~p\lambda)   &c-\iota_{2}\leq \iota_3\leq c-1 & 0\leq \iota_4\leq c-\iota_{3}&\iota_5=c-(\iota_{3}+	\iota_{4})&\iota_{6}=\iota_{4}\\
					&\iota_{3}^{'}=\iota_3&\iota_{4}^{'}=\iota_{4}&\iota_{5}^{'}=\iota_{5}&\iota_{6}^{'}=\iota_{6}\\			
					-\bar{\delta}_{\iota_{2}0}(\iota_{4}\mu+&&&&\\
					\bar{\delta}_{\iota_{6}N}	(\lambda+\iota_{1}\theta)&&&&\\
					+\eta+\delta_{\iota_{6}N}~p\lambda)   &c-\iota_{2}\leq \iota_3\leq c-1 &  \iota_4=\iota_{2}&\iota_5=0&\iota_{2}+1\leq \iota_6\leq N\\
					&\iota_{3}^{'}=\iota_3&\iota_{4}^{'}=\iota_{4}&\iota_{5}^{'}=\iota_{5}&\iota_{6}^{'}=\iota_{6}\\	
					0, & \mbox{otherwise.}
				\end{array}\right.\\*[0.25cm]
			\end{eqnarray}
		\end{adjustwidth}
		For $\iota_{2}\in c,c+1,\cdots,S$
	\begin{eqnarray}\label{z16}\nonumber
		\mathbb{G}_{c} =  \left\{ \begin{array}{lllllll}		
			-(\iota_{4}\mu+\delta_{\iota_{3}0}~\eta&&&&\\
			+\lambda+\iota_{1}\theta)   &0\leq \iota_3\leq c-1 & 0\leq \iota_4\leq c-\iota_{3}&\iota_5=c-(\iota_{3}+	\iota_{4})& \iota_6=\iota_4\\
			&\iota_{3}^{'}=\iota_3&\iota_{4}^{'}=\iota_{4}&\iota_{5}^{'}=\iota_{5}&\iota_{6}^{'}=\iota_{6}\\			
			-(\iota_{4}\mu+\eta&&&&\\
			+\bar{\delta}_{\iota_{6}N}(\lambda+\iota_{1}\theta)&&&&\\+\delta_{\iota_{6}N}~p\lambda)   &0\leq \iota_3\leq c-1&  \iota_4=c-\iota_{3}&\iota_5=0&c-\iota_{3}+1\leq \iota_6\leq N\\
			&\iota_{3}^{'}=\iota_3&\iota_{4}^{'}=\iota_{4}&\iota_{5}^{'}=\iota_{5}&\iota_{6}^{'}=\iota_{6}\\			
			\lambda	&0\leq \iota_3\leq c-1&  0\leq \iota_4\leq c-\iota_{3}&\iota_5=c-(\iota_{3}+	\iota_{4})&\iota_6=\iota_4\\
			&\iota_{3}^{'}=\iota_3&\iota_{4}^{'}=\iota_{4}&\iota_{5}^{'}=\iota_{5}&\iota_{6}^{'}=\iota_{6}+1\\
			\lambda	&0\leq \iota_3\leq c-1& \iota_4= c-\iota_{3}&\iota_5=0&c-\iota_{3}+1\leq \iota_6\leq N-1\\
			&\iota_{3}^{'}=\iota_3&\iota_{4}^{'}=\iota_{4}&\iota_{5}^{'}=\iota_{5}&\iota_{6}^{'}=\iota_{6}+1\\	
			0, & \mbox{otherwise,}
		\end{array}\right.\\*[0.25cm]
	\end{eqnarray}
	\end{adjustwidth}
\end{enumerate}
The ergodicity of the Markov chain $P(t)$ is assumed to  $\rho=\dfrac{\lambda}{c\mu}<1\text{(See ~\cite{arpo})}$.

\subsection{Neuts and Rao Matrix Geometric Approximation}
\indent \indent As the assumed Markov chain $\{P(t),t\geq 0\}$ has the structure of a level-dependent QBD process, it is hard to find the steady-state probability vector analytically. So we truncate the orbit level at the point $M$. The truncation point is defined as the point $M$, where $M \geq 1$ in which the system changes itself from level-dependent QBD to level-independent QBD when $\iota_1<M$ and  $\iota_1\geq M$, respectively. More explicitly, by setting $\mathbb{H}_{\iota_10} = \mathbb{H}_{M0}$ and $\mathbb{H}_{\iota_1 1}=\mathbb{H}_{M1}$ for all $\iota_1 \geq M$, we determine the equilibrium for the considered system. The choice of the truncated value $M$ is determined according to Marcel F Neuts and B M Rao \cite{Neu1} and Srinivas R Chakravarthy et.al \cite{Chak21}.

The determination of the modified infinitesimal generator structure of $H$ is based on the Neuts and Rao matrix geometric approximation method \cite{Neu1} and it is determined as follows:

\begin{eqnarray*}
	\hat{H} & = & \left(
	\begin{array}{cccccccccccccccccccccc}
		{\mathbb{H}_{01}} & \mathbb{H}_{0} & \mathbf{0} & \mathbf{0} & \mathbf{0} &\cdots & \mathbf{0} &\mathbf{0} &\mathbf{0} &\mathbf{0} &\mathbf{0} & \cdots\\
		\mathbb{H}_{10} & \mathbb{H}_{11} & \mathbb{H}_{0} & \mathbf{0} &  \mathbf{0} & \cdots & \mathbf{0} &\mathbf{0} &\mathbf{0} &\mathbf{0} &\mathbf{0} &\cdots\\
		\mathbf{0} & \mathbb{H}_{20} & \mathbb{H}_{21} & \mathbb{H}_{0} & \mathbf{0} & \cdots&\mathbf{0} &\mathbf{0} &\mathbf{0} &\mathbf{0} &\mathbf{0} &\cdots\\
		\vdots & \vdots & \vdots & \vdots & \vdots & \ddots&\vdots&\vdots&\vdots&\vdots&\vdots&\ddots\\
		\mathbf{0} & \mathbf{0} & \mathbf{0} &  \mathbf{0} & \mathbf{0} &  \cdots & \mathbb{H}_{M0} & \mathbb{H}_{M1} & \mathbb{H}_{0} & \mathbf{0} & \mathbf{0} & \cdots\\
		\mathbf{0} & \mathbf{0} &  \mathbf{0} & \mathbf{0} &  \mathbf{0} & \cdots & \mathbf{0} & \mathbb{H}_{M0} & \mathbb{H}_{M1} & \mathbb{H}_{0} & \mathbf{0} & \cdots\\
		\vdots & \vdots & \vdots & \vdots & \vdots & \ddots&\vdots&\vdots&\vdots&\vdots&\vdots&\ddots\\
	\end{array}\right),
\end{eqnarray*}
\begin{thm}		
	The probability vector $\phi$ in steady-state, which is related to the finite generator matrix $\mathbb{H}_{M}$ at truncation point $M$, can be obtained from the following equation: $\mathbb{H}_{M} \phi = 0$, where $\mathbb{H}_{M}=\mathbb{H}_{M0} + \mathbb{H}_{M1} + \mathbb{H}_{0}$ and the vector $\phi$ as converted as follows:
	\begin{eqnarray}\label{w1}
		{\phi^{( \iota_2)}} & = & {\phi^{(Q)}} {\nabla_{ \iota_2}}, \quad  \iota_2 = 0^*,Q^*,0,1, \ldots, S.
	\end{eqnarray}
	where
	\begin{adjustwidth}{-2cm}{1cm}
		\begin{eqnarray*}
			\nabla_{ \iota_2}  = \left\{ \begin{array}{lll}
				(-1)^{Q} 	(\mathbb{F}_{c+1}\mathbb{	\hat{G}}_c^{-1})^{Q-c}\left(\prod\limits_{i=1}^{c-1}\mathbb{F}_{i+1}	\mathbb{\hat{G}}_{i}^{-1}\right) \mathbb{J}_{\iota_{2}}\mathbb{\hat{G}}_{\iota_{2}}^{-1}& \iota_{2}=0^*\\
				(-1)^{Q+1} 	(\mathbb{F}_{c+1}\mathbb{\hat{G}}_c^{-1})^{Q-c}\left(\prod\limits_{i=1}^{c-1}\mathbb{F}_{i+1}\mathbb{\hat{G}}_{i}^{-1}\right) \mathbb{J}_{0^*}\mathbb{\hat{G}}_{0^*}^{-1} \mathbb{J}_{\iota_{2}}\mathbb{\hat{G}}_{\iota_{2}}^{-1}& \iota_{2}=Q^*\\
				(-1)^{Q} 	(\mathbb{F}_{c+1}\mathbb{\hat{G}}_c^{-1})^{Q-c}\left(\prod\limits_{i=1}^{c-1}\mathbb{F}_{i+1}\mathbb{\hat{G}}_{i}^{-1}\right)\left(\mathbb{F}_{1}G_{0}^{-1}-\mathbb{J}_{0^*}\mathbb{\hat{G}}_{0^*}^{-1}\mathbb{J}_{0}\mathbb{\hat{G}}_{0}^{-1}\right)  & \iota_{2}=0\\	
				(-1)^{Q-\iota_{2}} 	(\mathbb{F}_{c+1}\mathbb{\hat{G}}_c^{-1})^{Q-c}\left(\prod\limits_{i=\iota_{2}}^{c-1}\mathbb{F}_{i+1}\mathbb{\hat{G}}_{i}^{-1}\right)& \iota_{2}\in 1,2,\cdots,c-1\\
				(-1)^{Q-\iota_{2}} 	(\mathbb{F}_{c+1}\mathbb{\hat{G}}_c^{-1})^{Q-\iota_{2}} & \iota_{2}\in c,c+1,\cdots,S\\	
				I & \iota_{2}=Q\\
				(-1)^{Q-(j-1)} 	(\mathbb{F}_{c+1}\mathbb{\hat{G}}_c^{-1})^{Q-s}&\\
				\Biggl[\sum\limits_{ x=j+1}^{c}\left((\mathbb{F}_{c+1}\mathbb{\hat{G}}_c^{-1})^{s-c} \left(\prod\limits_{i=c}^{k}\mathbb{F}_{i}\mathbb{\hat{G}}_{i-1}^{-1}\right)	\mathbb{L}_{k}\mathbb{\hat{G}}_c^{-1}(\mathbb{F}_{c+1}\mathbb{\hat{G}}_c^{-1})^{x-(j+1)}\right)+&\\ \left( \sum\limits_{x=0}^{s-c}(\mathbb{F}_{c+1}\mathbb{\hat{G}}_c^{-1})^{s-c-x}\mathbb{L}_{c}\mathbb{\hat{G}}_c^{-1}(\mathbb{F}_{c+1}\mathbb{\hat{G}}_c^{-1})^{x}\right)(\mathbb{F}_{c+1}\mathbb{\hat{G}}_c^{-1})^{c-j }\Biggr] & \iota_{2}\in Q+j,\\ & j\in 1,2,\cdots,c-1\\
				(-1)^{Q-(j-1)} 	(\mathbb{F}_{c+1}\mathbb{\hat{G}}_c^{-1})^{Q-s}\left\{\sum\limits_{ x=0}^{s-j}(\mathbb{F}_{c+1}\mathbb{\hat{G}}_c^{-1})^{s-j-x} 	\mathbb{L}_{c}\mathbb{\hat{G}}_c^{-1}(\mathbb{F}_{c+1}\mathbb{\hat{G}}_c^{-1})^{x}\right\}& \iota_{2}\in  Q+j, \\& j\in c,c+1,\cdots,S\\
				0, &\mbox{otherwise}
			\end{array}\right.\\*[0.15cm]\nonumber		
		\end{eqnarray*}
	\end{adjustwidth}
	and
	${\phi^{(Q)}}$ is obtained by solving the equation $ \phi^{(Q^*)}\mathbb{J}_{\iota_2}+\phi^{(0)}\mathbb{L}_{0}+	\phi^{(\iota_{2})}\hat{\mathbb{G}}_{c}+	\phi^{(\iota_{2}+1)}\mathbb{F}_{c+1} ={\bf{0}} $ 
	and $ \sum\limits_{ \iota_2=0^*,Q^*} \phi^{( \iota_2)}{\bf e}+\sum\limits_{ \iota_2=0}^{S} \phi^{( \iota_2)}{\bf e}=1$
\end{thm}
\begin{proof}
	We have
	\begin{eqnarray}
		\phi ~\mathbb{H}_M= {\bf{0}} \ \ \text{and} \ \ \phi ~{\bf {e}}= 1
	\end{eqnarray}
	where
	\begin{eqnarray*}
		\left[\mathbb{H}_{M}\right]_{\iota_{2},\iota_{2}^{'}} & = &  \left\{ \begin{array}{llll}
			\mathbb{J}_{\iota_{2}^{'}} &  \iota_{2}^{'}=0^*& \iota_{2}=1\\
			\mathbb{J}_{\iota_{2}^{'}}  &  \iota_{2}^{'}=Q^*,0& \iota_{2}=0^*\\
			\mathbb{J}_{\iota_{2}^{'}}  &  \iota_{2}^{'}=Q& \iota_{2}=Q^*\\	
			\mathbb{F}_{\iota_{2}}& \iota_{2}^{'}=\iota_{2}-1 & \iota_{2}\in 1,2,\cdots,c\\
			\mathbb{F}_{c+1} & \iota_{2}^{'}=\iota_{2}-1 & \iota_{2}\in c+1,c+2,\cdots,S\\	
			\mathbb{L}_{\iota_{2}}& \iota_{2}^{'}=\iota_{2}+Q & \iota_{2}\in 0,1,2,\cdots,c\\
			\mathbb{L}_{c}& \iota_{2}^{'}=\iota_{2}+Q & \iota_{2}\in c+1,c+2,\cdots,s\\
			\hat{\mathbb{G}}_{\iota_{2}^{'} }& \iota_{2}^{'}=\iota_{2} & \iota_{2}\in 0^*,Q^*,0,1,\cdots,c-1\\
			\hat{\mathbb{G}}_c & \iota_{2}^{'}=\iota_{2} & \iota_{2}\in c,c+1,\cdots,S\\
			0, &\mbox{otherwise}
		\end{array}\right.\\*[0.15cm]\nonumber
	\end{eqnarray*}
where
	For $\iota_{2}\in 0^*,Q^*$
\begin{eqnarray}\nonumber
	\mathbb{\hat{G}}_{\iota_{2}} & = & \left\{ \begin{array}{llllll}
		-(M\theta+\bar{\delta}_{\iota_{6}N}~\lambda+\delta_{\iota_{2}0^*}~\beta&&&&\\+\delta_{\iota_{2}0^*}\bar{\delta}_{\iota_{6}N}~\eta 
		+\delta_{\iota_{2}Q^*}~\eta)	&\iota_3=c & \iota_4=0&\iota_5=0&\iota_6=0,1,\cdots,N\\
		&\iota_{3}^{'}=c & \iota_{4}^{'}=\iota_4&\iota_{5}^{'}=\iota_5&\iota_{6}^{'}=\iota_{6} \\
		\lambda+M\theta	&\iota_3=c & \iota_4=0&\iota_5=0&\iota_6=0,1,\cdots,N-1\\
		&\iota_{3}^{'}=c & \iota_{4}^{'}=\iota_4&\iota_{5}^{'}=\iota_5&\iota_{6}^{'}=\iota_{6}+1 \\
		0, & \mbox{otherwise,}
	\end{array}\right.\\*[0.25cm]
\end{eqnarray}
For $\iota_{2}\in 0,1,\cdots,c-1$
\begin{adjustwidth}{-.8cm}{-1cm}
	\begin{eqnarray}\nonumber
		\mathbb{\hat{G}}_{\iota_{2} } =  \left\{ \begin{array}{lllllll}
			\lambda+M\theta		&0\leq \iota_3\leq c-(\iota_{2}+1) & \iota_4=\iota_{2}&\iota_5=c-(\iota_{3}+	\iota_{4})&\iota_{4}+	\iota_{5}\leq \iota_6\leq N-1\\
			&\iota_{3}^{'}=\iota_3&\iota_{4}^{'}=\iota_{4}&\iota_{5}^{'}=\iota_{5}&\iota_{6}^{'}=\iota_{6}+1\\
			\bar{\delta}_{\iota_{2}0}(	\lambda+M\theta	)	&c-\iota_{2}\leq \iota_3\leq c-1 & 0\leq \iota_4\leq c-\iota_{3}&\iota_5=c-(\iota_{3}+	\iota_{4})&\iota_6=	\iota_{4}\\
			&\iota_{3}^{'}=\iota_3&\iota_{4}^{'}=\iota_{4}&\iota_{5}^{'}=\iota_{5}&\iota_{6}^{'}=\iota_{6}+1\\	
			\bar{\delta}_{\iota_{2}0}(	\lambda+M\theta	)  &c-\iota_{2}\leq \iota_3\leq c-1 & \iota_4=\iota_{2}&\iota_5=0&c-\iota_{3}+1\leq \iota_6\leq N-1\\
			&\iota_{3}^{'}=\iota_3&\iota_{4}^{'}=\iota_{4}&\iota_{5}^{'}=\iota_{5}&\iota_{6}^{'}=\iota_{6}+1\\
			-(\iota_{4}\mu&&&&\\
			+\bar{\delta}_{\iota_{6}N}(\lambda+M\theta)&&&&\\
			+\bar{\delta}_{\iota_{3}0}\bar{\delta}_{\iota_{6}\iota_{5}}\eta)   &0\leq \iota_3\leq c-(\iota_{2}+1) & \iota_4=\iota_{2}&\iota_5=c-(\iota_{3}+	\iota_{4})&\iota_{4}+\iota_{5}\leq \iota_6\leq N\\
			&\iota_{3}^{'}=\iota_3&\iota_{4}^{'}=\iota_{4}&\iota_{5}^{'}=\iota_{5}&\iota_{6}^{'}=\iota_{6}\\	
		\end{array}\right.\\*[0.25cm]
	\end{eqnarray}
\end{adjustwidth}
			\begin{adjustwidth}{-.3cm}{-1cm}
				\begin{eqnarray}\nonumber		
						\mathbb{\hat{G}}_{\iota_{2} } =  \left\{ \begin{array}{lllllll}
			-\bar{\delta}_{\iota_{2}0}(\iota_{4}\mu+&&&&\\
			(\lambda+M\theta)+&&&&\\
			\bar{\delta}_{\iota_{2}(c-1)}\bar{\delta}_{\iota_{3}(c-1)}\eta)   &c-\iota_{2}\leq \iota_3\leq c-1 & 0\leq \iota_4\leq c-\iota_{3}&\iota_5=c-(\iota_{3}+	\iota_{4})&\iota_{6}=\iota_{4}\\
			&\iota_{3}^{'}=\iota_3&\iota_{4}^{'}=\iota_{4}&\iota_{5}^{'}=\iota_{5}&\iota_{6}^{'}=\iota_{6}\\			
			-\bar{\delta}_{\iota_{2}0}(\iota_{4}\mu+&&&&\\
			\bar{\delta}_{\iota_{6}N}	(\lambda+M\theta)&&&&\\
			+\eta)   &c-\iota_{2}\leq \iota_3\leq c-1 &  \iota_4=\iota_{2}&\iota_5=0&\iota_{2}+1\leq \iota_6\leq N\\
			&\iota_{3}^{'}=\iota_3&\iota_{4}^{'}=\iota_{4}&\iota_{5}^{'}=\iota_{5}&\iota_{6}^{'}=\iota_{6}\\	
			0, & \mbox{otherwise.}
		\end{array}\right.\\*[0.25cm]
	\end{eqnarray}
\end{adjustwidth}
For $\iota_{2}\in c,c+1,\cdots,S$
\begin{adjustwidth}{-2cm}{-1cm}
\begin{eqnarray}\nonumber
	\mathbb{\hat{G}}_{c} =  \left\{ \begin{array}{lllllll}		
		-(\iota_{4}\mu+\delta_{\iota_{3}0}~\eta&&&&\\
		+\lambda+M\theta)   &0\leq \iota_3\leq c-1 & 0\leq \iota_4\leq c-\iota_{3}&\iota_5=c-(\iota_{3}+	\iota_{4})& \iota_6=\iota_4\\
		&\iota_{3}^{'}=\iota_3&\iota_{4}^{'}=\iota_{4}&\iota_{5}^{'}=\iota_{5}&\iota_{6}^{'}=\iota_{6}\\			
		-(\iota_{4}\mu+\eta&&&&\\
		+\bar{\delta}_{\iota_{6}N}(\lambda+M\theta))   &0\leq \iota_3\leq c-1&  \iota_4=c-\iota_{3}&\iota_5=0&c-\iota_{3}+1\leq \iota_6\leq N\\
		&\iota_{3}^{'}=\iota_3&\iota_{4}^{'}=\iota_{4}&\iota_{5}^{'}=\iota_{5}&\iota_{6}^{'}=\iota_{6}\\			
		\lambda+M\theta	&0\leq \iota_3\leq c-1&  0\leq \iota_4\leq c-\iota_{3}&\iota_5=c-(\iota_{3}+	\iota_{4})&\iota_6=\iota_4\\
		&\iota_{3}^{'}=\iota_3&\iota_{4}^{'}=\iota_{4}&\iota_{5}^{'}=\iota_{5}&\iota_{6}^{'}=\iota_{6}+1\\
		\lambda+M\theta	&0\leq \iota_3\leq c-1& \iota_4= c-\iota_{3}&\iota_5=0&c-\iota_{3}+1\leq \iota_6\leq N-1\\
		&\iota_{3}^{'}=\iota_3&\iota_{4}^{'}=\iota_{4}&\iota_{5}^{'}=\iota_{5}&\iota_{6}^{'}=\iota_{6}+1\\	
		0, & \mbox{otherwise,}
	\end{array}\right.\\*[0.25cm]
\end{eqnarray}
\end{adjustwidth}
	The equation $	\phi ~\mathbb{H}_M= \bf{0} $ yields the following set of equations:
	\begin{align}
		\phi^{( 0^*)} \hat{\mathbb{G}}_{\iota_2}+\phi^{( 1)}\mathbb{J}_{\iota_2}&={\bf{0}},\iota_2=0^*  \\
		\phi^{( 0^*)} \mathbb{J}_{\iota_2}+\phi^{( Q^*)}\hat{\mathbb{G}}_{\iota_2}&={\bf{0}},\iota_2=Q^*  \\
		\phi^{( 0^*)} \mathbb{J}_{\iota_2}+\phi^{(0)}\hat{\mathbb{G}}_{\iota_2}+	\phi^{(1)}\mathbb{F}_{\iota_2+1}&={\bf{0}},\iota_2=0   \\
		\phi^{(\iota_{2})}\hat{\mathbb{G}}_{\iota_{2}}+\phi^{(\iota_{2}+1)}	\mathbb{F}_{\iota_{2}+1}&={\bf{0}} ,  \iota_2=1,2,\cdots,c\\	
		\phi^{(\iota_{2})}\hat{\mathbb{G}}_{c}+	\Phi^{(\iota_{2}+1)}\mathbb{F}_{c+1}&={\bf{0}} ,  \iota_2=c+1,c+2,\cdots,Q-1\\
		\phi^{(Q^*)}\mathbb{J}_{\iota_2}+\phi^{(0)}\mathbb{L}_{0}+	\phi^{(\iota_{2})}\hat{\mathbb{G}}_{c}+	\phi^{(\iota_{2}+1)}\mathbb{F}_{c+1}&={\bf{0}} ,  \iota_2=Q\\	
		\phi^{( \iota_2-Q)}\mathbb{L}_{\iota_2-Q}+\phi^{( \iota_2)}\hat{\mathbb{G}}_{c}+	\phi^{( \iota_2+1)}\mathbb{F}_{c+1}&={\bf{0}} ,   \iota_2=Q+1,Q+2,\cdots,Q+c-2\\	
		\phi^{( \iota_2-Q)}\mathbb{L}_{c}+\phi^{( \iota_2)}\hat{\mathbb{G}}_{c}+	\phi^{( \iota_2+1)}\mathbb{F}_{c+1}&={\bf{0}} ,   \iota_2=Q+c-1,Q+c-2,\cdots,S-1\\
		\phi^{( \iota_2-Q)}\mathbb{L}_{c}+\phi^{( \iota_2)}\hat{\mathbb{G}}_{c}&={\bf{0}} ,   \iota_2=S
	\end{align}
	We arrive at the stated solution by iteratively solving the aforementioned system of equations with its normalization condition $\phi ~\mathbf {e} = 1$.
\end{proof} 
It should be noted that the associated state space of the modified infinitesimal generator matrix $\hat{H}$ is the same as $H$. However the ergodicity condition of the original Markov chain $\lambda < c\mu$ does not hold for the truncated Markov chain. Because the retrial rate is constant from the level $M$ up. So following the general theory for QBD process, the ergodicity condition of the truncated Markov chain is derived as given in the Theorem \ref{t1}.
\begin{thm}\label{t1}
	The stability condition for the modified infinitesimal generator matrix, $\hat{H}$ is given by the inequality 
	\begin{equation}\label{e1}
		z_1 p\lambda <z_2 M\theta
	\end{equation}
	\begin{adjustwidth}{-1cm}{1cm}
		\begin{align*}
			\text{where}~~ z_1&=\phi^{( 0,c,0,0,N)}+\phi^{(Q,c,0,0,N)}+\sum\limits_{ \iota_2=0}^{c-1}\sum\limits_{ \iota_3=0}^{c-(\iota_2+1)}\phi^{(\iota_2, \iota_3, \iota_2,c-( \iota_2+ \iota_3),N)}\\&+\sum\limits_{ \iota_2=1}^{c-1}\sum\limits_{ \iota_3=c-\iota_2}^{c-1}\phi^{( \iota_2, \iota_3, c-\iota_3,0,N)}+\sum\limits_{ \iota_2=c}^{S}\sum\limits_{ \iota_3=0}^{c-1}\sum\limits_{ \iota_4=0}^{c-\iota_3}\phi^{( \iota_2,\iota_3,\iota_4,c-( \iota_3+ \iota_4),N)}\\&\text{and}~ z_2=\sum\limits_{ \iota_6=0}^{N-1}\phi^{( 0,c,0,0,\iota_6)}+\sum\limits_{ \iota_6=0}^{N-1}\phi^{( Q,c,0,0,\iota_6)}+\sum\limits_{ \iota_2=0}^{c-1}\sum\limits_{ \iota_3=0}^{ c-(\iota_2+1)}\sum\limits_{ \iota_6=c-\iota_3}^{N-1}\phi^{( \iota_2, \iota_3, \iota_2,c-( \iota_2+ \iota_3),\iota_6)}\\&+\sum\limits_{ \iota_2=0}^{c-1}\sum\limits_{ \iota_3=c-\iota_{2}}^{ c-1}\sum\limits_{ \iota_4=0}^{c-\iota_{3}}\phi^{( \iota_2, \iota_3, \iota_4,c-( \iota_3+ \iota_4),\iota_4)}+\sum\limits_{ \iota_2=0}^{c-1}\sum\limits_{ \iota_3=c-\iota_{2}}^{ c-1}\sum\limits_{ \iota_6=c-\iota_3+1}^{N-1}\phi^{( \iota_2, \iota_3, \iota_2,0,\iota_6)}\\&+\sum\limits_{ \iota_2=c}^{S}\sum\limits_{ \iota_3=0}^{ c-1}\sum\limits_{ \iota_4=0}^{c-\iota_{3}}\phi^{(\iota_2, \iota_3, \iota_4,c-( \iota_3+ \iota_4),\iota_4)}+\sum\limits_{ \iota_2=c}^{S}\sum\limits_{ \iota_3=0}^{c-1}\sum\limits_{ \iota_6=\iota_4+1}^{N-1}\phi^{(\iota_2, \iota_3, c- \iota_3,0,\iota_6)}.
		\end{align*}
	\end{adjustwidth}
	
\end{thm}
\begin{proof}
	According to Marcel F Neuts \cite{Neu}, the steady-state probability vector $\Phi$ for the modified infinitesimal generator matrix exists if and only if the following inequality is satisfied:
	$$\phi \mathbb{
		H}_{0} \textbf{e} < \phi \mathbb{H}_{M0} \textbf{e}.$$ 
	
	Now, writing all the matrices and a stationary vector to the generator matrix $\mathbb{H}_M$ is given by	
	\begin{align*}
		&p\lambda\phi^{(0,c,0,0,N)}+p\lambda \phi^{(Q,c,0,0,N)}+\sum\limits_{ \iota_2=0}^{c-1}\sum\limits_{ \iota_3=0}^{c-(\iota_2+1)}p\lambda \phi^{(\iota_2, \iota_3, \iota_2,c-( \iota_2+ \iota_3),N)}\\&+\sum\limits_{ \iota_2=1}^{c-1}\sum\limits_{ \iota_3=c-\iota_2}^{c-1}p\lambda \phi^{( \iota_2, \iota_3, c-\iota_3,0,N)}
		+\sum\limits_{ \iota_2=c}^{S}\sum\limits_{ \iota_3=0}^{c-1}\sum\limits_{ \iota_4=0}^{c-\iota_3}p\lambda \phi^{( \iota_2,\iota_3,\iota_4,c-( \iota_3+ \iota_4),N)} \\&<\sum\limits_{ \iota_6=0}^{N-1}M\theta\phi^{( 0,c,0,0,\iota_6)}+\sum\limits_{ \iota_6=0}^{N-1}M\theta \Phi^{(M, Q,c,0,0,\iota_6)}\\&+\sum\limits_{ \iota_2=0}^{c-1}\sum\limits_{ \iota_3=0}^{ c-(\iota_2+1)}\sum\limits_{ \iota_6=c-\iota_3}^{N-1}M\theta \phi^{(\iota_2, \iota_3, \iota_2,c-( \iota_2+ \iota_3),\iota_6)}+\sum\limits_{ \iota_2=0}^{c-1}\sum\limits_{ \iota_3=c-\iota_{2}}^{ c-1}\sum\limits_{ \iota_4=0}^{c-\iota_{3}}M\theta \phi^{( \iota_2, \iota_3, \iota_4,c-( \iota_3+ \iota_4),\iota_4)}\\&+\sum\limits_{ \iota_2=0}^{c-1}\sum\limits_{ \iota_3=c-\iota_{2}}^{ c-1}\sum\limits_{ \iota_6=c-\iota_3+1}^{N-1}M\theta \phi^{(\iota_2, \iota_3, \iota_2,0,\iota_6)}+\sum\limits_{ \iota_2=c}^{S}\sum\limits_{ \iota_3=0}^{ c-1}\sum\limits_{ \iota_4=0}^{c-\iota_{3}}M\theta \phi^{( \iota_2, \iota_3, \iota_4,c-( \iota_3+ \iota_4),\iota_4)}\\&+\sum\limits_{ \iota_2=c}^{S}\sum\limits_{ \iota_3=0}^{c-1}\sum\limits_{ \iota_6=\iota_4+1}^{N-1}M\theta\phi^{( \iota_2, \iota_3, c- \iota_3,0,\iota_6)}.
	\end{align*}
	Hence, we get the desired inequality \eqref{e1}.
\end{proof}
\subsection{Limiting Probability Distribution}
\indent \indent The generator matrix $H$ has a steady state probability vector $\mathbf{\Phi} = (\Phi^{(0)}, \Phi^{(1)}, \Phi^{(2)}, \cdots)$, and it satisfies the stability condition. Thus, the Markov process $${(P_1(t), P_2(t), P_3(t), P_4(t), P_5(t), P_6(t)), t \geq 0}$$ with state space $\Omega$ is regular. Consequently, the limiting probability distribution is given by:
\begin{eqnarray*}
	{\Phi}^{(\iota_1,\iota_2,\iota_3,\iota_4,\iota_5,\iota_6)} = \lim_{t \rightarrow \infty} Pr\left[P_1(t) = \iota_1, P_2(t) = \iota_2, P_3(t)=\iota_3, P_4(t) = \iota_4, P_5(t) = \iota_5, P_6(t)=\iota_6~|\right. \\ P_1(0)=0, P_2(0)=0, \left.P_3(0)=0,P_4(0)=0, P_5(0)=0, P_6(0)=0 \right].
\end{eqnarray*}
exists and is independent of the initial state.
\subsection{Computation of R-Matrix}
%
\begin{thm}\label{t10}
	Let the matrix  $\mathbb{R}$ can be determined by the matrix quadratic equation
	\begin{eqnarray}\label{e5}
		\mathbb{R}^{2}\mathbb{H}_{M0}+\mathbb{R}\mathbb{H}_{M1}+\mathbb{H}_{0} =\textbf{0}.
	\end{eqnarray} 
	The square matrix $\mathbb{R}$ is of dimension $Sc(N + 1) +(c+2)N-\dfrac{2c^3+3c^2-5c-12}{6}$ and is defined by
	\begin{align} \label{e6}
		\mathbb{R}=\bordermatrix{& 0^* & Q^* & 0& 1&\cdots & S-1 & S \cr
			0^*& \mathfrak{A}^{(0^*,0^*)} & \mathfrak{A}^{(0^*,Q^* )} & \mathfrak{A}^{(0^*,0)}  & \mathfrak{A}^{(0^*,1)} &\cdots & \mathfrak{A}^{(0^*,S-1)} &\mathfrak{A}^{(0^*,S)} \cr
			Q^*  & \mathfrak{A}^{(Q^*,0^*)} & \mathfrak{A}^{(Q^*,Q^*) } &\mathfrak{A}^{(Q^*,0)}  & \mathfrak{A}^{(Q^*,1)}&\cdots & \mathfrak{A}^{(Q^*,S-1)} & \mathfrak{A}^{(Q^*,S)} \cr
			0& \mathfrak{A}^{(0,0^*)} & \mathfrak{A}^{(0,Q^*) } &\mathfrak{A}^{(0,0)}  &\mathfrak{A}^{(0,1)} &\cdots & \mathfrak{A}^{(0,S-1)} & \mathfrak{A}^{(0,S)} \cr
			1& \mathfrak{A}^{(1,0^*)} & \mathfrak{A}^{(1,Q^*) } &\mathfrak{A}^{(1,0)}  &\mathfrak{A}^{(1,1)}  &\cdots & \mathfrak{A}^{(1,S-1)} & \mathfrak{A}^{(1,S)} \cr
			\vdots & \vdots & \vdots & \vdots  & \vdots  &\cdots  & \vdots & \vdots \cr
			S-1& \mathfrak{A}^{(S-1,0^*)} & \mathfrak{A}^{(S-1,Q^*) } &\mathfrak{A}^{(S-1,0)}  &\mathfrak{A}^{(S-1,1)}  &\cdots & \mathfrak{A}^{(S-1,S-1)} & \mathfrak{A}^{(S-1,S)} \cr
			S & \mathfrak{A}^{(S,0^*)} & \mathfrak{A}^{(S,Q^*) } &\mathfrak{A}^{(S,0)}  & \mathfrak{A}^{(S,1)}&\cdots & \mathfrak{A}^{(S,S-1)} & \mathfrak{A}^{(S,S)} \cr
		}
	\end{align}
	For $\iota_2=0^*,Q^*$ and $\iota^{\prime}_2=0^*,Q^*$
	$$\mathfrak{A}^{(\iota_2,\iota^{\prime}_2)}=\begin{pmatrix}
		0 & 0 & 0 & \cdots & 0 \\
		0 & 0 & 0 & \cdots & 0 \\
		\vdots & \vdots & \vdots & \vdots & \vdots \\
		e^{\iota_2,\iota^{\prime}_2}_0 & e^{\iota_2,\iota^{\prime}_2}_1 & e^{\iota_2,\iota^{\prime}_2}_2 & \cdots & e^{\iota_2,\iota^{\prime}_2}_N    
	\end{pmatrix}_{(N+1)\times(N+1)}$$
	For $\iota_2=0^*,Q^*$ and $\iota^{\prime}_2=0$
	\begin{equation*}
		\mathfrak{A}^{(\iota_2,\iota^{\prime}_2)}=\left[\begin{array}{c}
			x^{\iota_2,\iota^{\prime}_2}_{0},
			x^{\iota_2,\iota^{\prime}_2}_{1},
			x^{\iota_2,\iota^{\prime}_2}_{2}, 
			\cdots, 
			x^{\iota_2,\iota^{\prime}_2}_{c-1}
		\end{array}\right]
	\end{equation*}
	\qquad For $\iota_3=0,1,\cdots,c-1$
	$$x^{\iota_2,\iota^{\prime}_2}_{\iota_3}=\begin{pmatrix}
		0 & 0 & 0 & \cdots & 0 \\
		0 & 0 & 0 & \cdots & 0 \\
		\vdots & \vdots & \vdots & \vdots & \vdots \\
		e^{\iota_2,\iota^{\prime}_2}_1 & e^{\iota_2,\iota^{\prime}_2}_2 & e^{\iota_2,\iota^{\prime}_2}_3 & \cdots & e^{\iota_2,\iota^{\prime}_2}_{N-(c-(\iota_3+1))}   
	\end{pmatrix}_{(N+1)\times(N-(c-(\iota_3+1)))}$$
	For $\iota_2=0^*,Q^*$ and $\iota^{\prime}_2=1,2,\cdots,c-1$
	\begin{equation*}
		\mathfrak{A}^{(\iota_2,\iota^{\prime}_2)}=\left[\begin{array}{c}
			x^{\iota_2,\iota^{\prime}_2}_{0},
			x^{\iota_2,\iota^{\prime}_2}_{1},
			x^{\iota_2,\iota^{\prime}_2}_{2}, 
			\cdots, 
			x^{\iota_2,\iota^{\prime}_2}_{c-1}
		\end{array}\right]
	\end{equation*}
	\qquad For $\iota_3=0,1,\cdots,c-(\iota_2+1)$
	$$x^{\iota_2,\iota^{\prime}_2}_{\iota_3}=\begin{pmatrix}
		0 & 0 & 0 & \cdots & 0 \\
		0 & 0 & 0 & \cdots & 0 \\
		\vdots & \vdots & \vdots & \vdots & \vdots \\
		e^{\iota_2,\iota^{\prime}_2}_1 & e^{\iota_2,\iota^{\prime}_2}_2 & e^{\iota_2,\iota^{\prime}_2}_3 & \cdots & e^{\iota_2,\iota^{\prime}_2}_{N-(c-(\iota_3+1))}   
	\end{pmatrix}_{(N+1)\times(N-(c-(\iota_3+1)))}$$
	\qquad For $\iota_3=c-(\iota_2+1)+1,\cdots,c-1$
	$$ x^{(\iota_2,\iota^{\prime}_2)}_{\iota_3}=\begin{pmatrix}
		0 & 0 & 0 & \cdots & 0 \\
		0 & 0 & 0 & \cdots & 0 \\
		\vdots & \vdots & \vdots & \vdots & \vdots \\
		e^{\iota_2,\iota^{\prime}_2}_0 & e^{\iota_2,\iota^{\prime}_2}_1 & e^{\iota_2,\iota^{\prime}_2}_2 & \cdots & e^{\iota_2,\iota^{\prime}_2}_N    
	\end{pmatrix}_{(N+1)\times(N+1)}$$
	For $\iota_2=0^*,Q^*$ and $\iota^{\prime}_2=c,c+1,\cdots,S$
	\begin{equation*}
		\mathfrak{A}^{(\iota_2,\iota^{\prime}_2)}=\left[\begin{array}{c}
			x^{\iota_2,\iota^{\prime}_2}_{0},
			x^{\iota_2,\iota^{\prime}_2}_{1},
			x^{\iota_2,\iota^{\prime}_2}_{2}, 
			\cdots, 
			x^{\iota_2,\iota^{\prime}_2}_{c-1}
		\end{array}\right]
	\end{equation*}
	\qquad For $\iota_3=0,1,\cdots,c-1$
	$$x^{(\iota_2,\iota^{\prime}_2)}_{\iota_3}=\begin{pmatrix}
		0 & 0 & 0 & 0 & 0 \\
		0 & 0 & 0 & 0 & 0 \\
		\vdots & \vdots & \vdots & \vdots & \vdots \\
		e^{\iota_2,\iota^{\prime}_2}_0 & e^{\iota_2,\iota^{\prime}_2}_1 & e^{\iota_2,\iota^{\prime}_2}_2 & \cdots & e^{\iota_2,\iota^{\prime}_2}_N    
	\end{pmatrix}_{(N+1)\times(N+1)}$$
	For $\iota_2=0$ and $\iota^{\prime}_2=0^*,Q^*$
	\begin{equation*}
		\mathfrak{A}^{(\iota_2,\iota^{\prime}_2)}=\left[\begin{array}{c}
			x^{\iota_2,\iota^{\prime}_2}_{0}\\
			x^{\iota_2,\iota^{\prime}_2}_{1}\\
			x^{\iota_2,\iota^{\prime}_2}_{2}\\
			\cdots\\
			x^{\iota_2,\iota^{\prime}_2}_{c-1}
		\end{array}\right]
	\end{equation*}
	\qquad For $\iota_3=0,1,\cdots,c-1$
	$$x^{\iota_2,\iota^{\prime}_2}_{\iota_3}=\begin{pmatrix}
		0 & 0 & 0 & 0 & 0 \\
		0 & 0 & 0 & 0 & 0 \\
		\vdots & \vdots & \vdots & \vdots & \vdots \\
		e^{\iota_2,\iota^{\prime}_2}_0 & e^{\iota_2,\iota^{\prime}_2}_1 & e^{\iota_2,\iota^{\prime}_2}_3 & \dots & e^{\iota_2,\iota^{\prime}_2}_{N}   
	\end{pmatrix}_{(N-(c-(\iota_3+1)))\times (N+1)}$$
	For $\iota_2=1,2,\cdots,c-1$ and $\iota^{\prime}_2=0^*,Q^*$
	\begin{equation*}
		\mathfrak{A}^{(\iota_2,\iota^{\prime}_2)}=\left[\begin{array}{c}
			x^{\iota_2,\iota^{\prime}_2}_{0}\\
			x^{\iota_2,\iota^{\prime}_2}_{1}\\
			x^{\iota_2,\iota^{\prime}_2}_{2}\\
			\cdots\\
			x^{\iota_2,\iota^{\prime}_2}_{c-1}
		\end{array}\right]
	\end{equation*}
	\qquad For $\iota_3=0,1,\cdots,c-(\iota_2+1)$
	$$x^{\iota_2,\iota^{\prime}_2}_{\iota_3}=\begin{pmatrix}
		0 & 0 & 0 & \cdots & 0 \\
		0 & 0 & 0 & \cdots & 0 \\
		\vdots & \vdots & \vdots & \vdots & \vdots \\
		e^{\iota_2,\iota^{\prime}_2}_0 & e^{\iota_2,\iota^{\prime}_2}_1 & e^{\iota_2,\iota^{\prime}_2}_2 & \dots & e^{\iota_2,\iota^{\prime}_2}_{N}   
	\end{pmatrix}_{(N-(c-(\iota_3+1)))\times (N+1)}$$
	\qquad $\iota_3=c-(\iota_2+1)+1,\cdots,c-1$
	$$ x^{(\iota_2,\iota^{\prime}_2)}_{\iota_3}=\begin{pmatrix}
		0 & 0 & 0 & 0 & 0 \\
		0 & 0 & 0 & 0 & 0 \\
		\vdots & \vdots & \vdots & \vdots & \vdots \\
		e^{\iota_2,\iota^{\prime}_2}_0 & e^{\iota_2,\iota^{\prime}_2}_1 & e^{\iota_2,\iota^{\prime}_2}_2 & \dots & e^{\iota_2,\iota^{\prime}_2}_N    
	\end{pmatrix}_{(N+1)\times(N+1)}$$
	For $\iota_2=c,c+1,\cdots,S$ and $\iota^{\prime}_2=0^*,Q^*$
	\begin{equation*}
		\mathfrak{A}^{(\iota_2,\iota^{\prime}_2)}=\left[\begin{array}{c}
			x^{\iota_2,\iota^{\prime}_2}_{0}\\
			x^{\iota_2,\iota^{\prime}_2}_{1}\\
			x^{\iota_2,\iota^{\prime}_2}_{2}\\
			\cdots\\
			x^{\iota_2,\iota^{\prime}_2}_{c-1}
		\end{array}\right]
	\end{equation*}
	\qquad For $\iota_3=0,1,\cdots,c-1$
	$$x^{(\iota_2,\iota^{\prime}_2)}_{\iota_3}=\begin{pmatrix}
		0 & 0 & 0 & 0 & 0 \\
		0 & 0 & 0 & 0 & 0 \\
		\vdots & \vdots & \vdots & \vdots & \vdots \\
		e^{\iota_2,\iota^{\prime}_2}_0 & e^{\iota_2,\iota^{\prime}_2}_1 & e^{\iota_2,\iota^{\prime}_2}_2 & \dots & e^{\iota_2,\iota^{\prime}_2}_N    
	\end{pmatrix}_{(N+1)\times(N+1)}$$
	For $\iota_2=0$ and $\iota^{\prime}_2=0$
	$$	\mathfrak{A}^{(\iota_2,\iota^{\prime}_2)}=\begin{pmatrix}
		x^{\iota_2,\iota^{\prime}_2}_{0,0} & x^{\iota_2,\iota^{\prime}_2}_{0,1} & x^{\iota_2,\iota^{\prime}_2}_{0,2} & \cdots & x^{\iota_2,\iota^{\prime}_2}_{0,c-1}  \\
		x^{\iota_2,\iota^{\prime}_2}_{1,0} & x^{\iota_2,\iota^{\prime}_2}_{1,1} & x^{\iota_2,\iota^{\prime}_2}_{1,2} & \cdots & x^{\iota_2,\iota^{\prime}_2}_{1,c-1}   \\
		\vdots & \vdots & \vdots & \vdots & \vdots \\
		x^{\iota_2,\iota^{\prime}_2}_{c-1,0} & x^{\iota_2,\iota^{\prime}_2}_{c-1,1} & x^{\iota_2,\iota^{\prime}_2}_{c-1,2} & \cdots & x^{\iota_2,\iota^{\prime}_2}_{c-1,c-1}    
	\end{pmatrix}_{c\times c}$$
	\qquad For $\iota_3=0,1,\cdots,c-1$ and $\iota^{\prime}_3=0,1,\cdots,c-1$
	$$x^{\iota_2,\iota^{\prime}_2}_{\iota_3,\iota^{\prime}_3}=\begin{pmatrix}
		0 & 0 & 0 & 0 & 0 \\
		\vdots & \vdots & \vdots & \vdots & \vdots \\
		0 & 0 & 0 & 0 & 0 \\
		e^{\iota_2,\iota^{\prime}_2}_1 & e^{\iota_2,\iota^{\prime}_2}_2 & e^{\iota_2,\iota^{\prime}_2}_3 & \dots & e^{\iota_2,\iota^{\prime}_2}_{N-(c-(\iota^{\prime}_3+1))}  
	\end{pmatrix}_{(N-(c-(\iota_3+1)))\times (N-(c-(\iota^{\prime}_3+1)))}$$
	For $\iota_2=1,2,\cdots,c-1$ and $\iota^{\prime}_2=1,2,\cdots,c-1$
	$$	\mathfrak{A}^{(\iota_2,\iota^{\prime}_2)}=\begin{pmatrix}
		x^{\iota_2,\iota^{\prime}_2}_{0,0} & x^{\iota_2,\iota^{\prime}_2}_{0,1} & x^{\iota_2,\iota^{\prime}_2}_{0,2} & \cdots & x^{\iota_2,\iota^{\prime}_2}_{0,c-1}  \\
		x^{\iota_2,\iota^{\prime}_2}_{1,0} & x^{\iota_2,\iota^{\prime}_2}_{1,1} & x^{\iota_2,\iota^{\prime}_2}_{1,2} & \cdots & x^{\iota_2,\iota^{\prime}_2}_{1,c-1}   \\
		\vdots & \vdots & \vdots & \vdots & \vdots \\
		x^{\iota_2,\iota^{\prime}_2}_{c-1,0} & x^{\iota_2,\iota^{\prime}_2}_{c-1,1} & x^{\iota_2,\iota^{\prime}_2}_{c-1,2} & \cdots & x^{\iota_2,\iota^{\prime}_2}_{c-1,c-1}    
	\end{pmatrix}_{c\times c}$$
	\qquad For $\iota_3=0,1,\cdots,c-(\iota_2+1)$ and $\iota^{\prime}_3=0,1,\cdots,c-(\iota^{\prime}_2+1)$
	$$x^{\iota_2,\iota^{\prime}_2}_{\iota_3,\iota^{\prime}_3}=\begin{pmatrix}
		0 & 0 & 0 & 0 & 0 \\
		\vdots & \vdots & \vdots & \vdots & \vdots \\
		0 & 0 & 0 & 0 & 0 \\
		e^{\iota_2,\iota^{\prime}_2}_1 & e^{\iota_2,\iota^{\prime}_2}_2 & e^{\iota_2,\iota^{\prime}_2}_3 & \dots & e^{\iota_2,\iota^{\prime}_2}_{N-(c-(\iota^{\prime}_3+1))}  
	\end{pmatrix}_{(N-(c-(\iota_3+1)))\times (N-(c-(\iota^{\prime}_3+1)))}$$
	\qquad For $\iota_3=0,1,\cdots,c-(\iota_2+1)$ and $\iota^{\prime}_3=c-(\iota^{\prime}_2+1)+1,\cdots,c-1$
	$$x^{\iota_2,\iota^{\prime}_2}_{\iota_3,\iota^{\prime}_3}=\begin{pmatrix}
		0 & 0 & 0 & 0 & 0 \\
		\vdots & \vdots & \vdots & \vdots & \vdots \\
		0 & 0 & 0 & 0 & 0 \\
		e^{\iota_2,\iota^{\prime}_2}_0 & e^{\iota_2,\iota^{\prime}_2}_1 & e^{\iota_2,\iota^{\prime}_2}_2 & \dots & e^{\iota_2,\iota^{\prime}_2}_{N}  
	\end{pmatrix}_{(N-(c-(\iota_3+1)))\times (N+1)}$$
	\qquad For $\iota_3=c-(\iota^{\prime}_2+1)+1,\cdots,c-1$ and $\iota^{\prime}_3=0,1,\cdots,c-(\iota^{\prime}_2+1)$
	$$x^{\iota_2,\iota^{\prime}_2}_{\iota_3,\iota^{\prime}_3}=\begin{pmatrix}
		0 & 0 & 0 & 0 & 0 \\
		\vdots & \vdots & \vdots & \vdots & \vdots \\
		0 & 0 & 0 & 0 & 0 \\
		e^{\iota_2,\iota^{\prime}_2}_1 & e^{\iota_2,\iota^{\prime}_2}_2 & e^{\iota_2,\iota^{\prime}_2}_3 & \dots & e^{\iota_2,\iota^{\prime}_2}_{N-(c-(\iota^{\prime}_3+1}  
	\end{pmatrix}_{(N+1)\times (N-(c-(\iota^{\prime}_3+1)))}$$
	\qquad For $\iota_3=c-(\iota^{\prime}_2+1)+1,\cdots,c-1$ and $\iota^{\prime}_3=c-(\iota^{\prime}_2+1)+1,\cdots,c-1$
	$$x^{\iota_2,\iota^{\prime}_2}_{\iota_3,\iota^{\prime}_3}=\begin{pmatrix}
		0 & 0 & 0 & 0 & 0 \\
		\vdots & \vdots & \vdots & \vdots & \vdots \\
		0 & 0 & 0 & 0 & 0 \\
		e^{\iota_2,\iota^{\prime}_2}_0 & e^{\iota_2,\iota^{\prime}_2}_1 & e^{\iota_2,\iota^{\prime}_2}_2 & \dots & e^{\iota_2,\iota^{\prime}_2}_{N}  
	\end{pmatrix}_{(N+1)\times (N+1)}$$
	For $\iota_2=c,c+1,\cdots,S$ and $\iota^{\prime}_2=c,c+1,\cdots,S$
	$$	\mathfrak{A}^{(\iota_2,\iota^{\prime}_2)}=\begin{pmatrix}
		x^{\iota_2,\iota^{\prime}_2}_{0,0} & x^{\iota_2,\iota^{\prime}_2}_{0,1} & x^{\iota_2,\iota^{\prime}_2}_{0,2} & \cdots & x^{\iota_2,\iota^{\prime}_2}_{0,c-1}  \\
		x^{\iota_2,\iota^{\prime}_2}_{1,0} & x^{\iota_2,\iota^{\prime}_2}_{1,1} & x^{\iota_2,\iota^{\prime}_2}_{1,2} & \cdots & x^{\iota_2,\iota^{\prime}_2}_{1,c-1}   \\
		\vdots & \vdots & \vdots & \vdots & \vdots \\
		x^{\iota_2,\iota^{\prime}_2}_{c-1,0} & x^{\iota_2,\iota^{\prime}_2}_{c-1,1} & x^{\iota_2,\iota^{\prime}_2}_{c-1,2} & \cdots & x^{\iota_2,\iota^{\prime}_2}_{c-1,c-1}    
	\end{pmatrix}_{c\times c}$$
	\qquad For $\iota_3=0,1,\cdots,c-1$ and $\iota^{\prime}_3=0,1,\cdots,c-1$
	$$x^{\iota_2,\iota^{\prime}_2}_{\iota_3,\iota^{\prime}_3}=\begin{pmatrix}
		0 & 0 & 0 & 0 & 0 \\
		\vdots & \vdots & \vdots & \vdots & \vdots \\
		0 & 0 & 0 & 0 & 0 \\
		e^{\iota_2,\iota^{\prime}_2}_0 & e^{\iota_2,\iota^{\prime}_2}_1 & e^{\iota_2,\iota^{\prime}_2}_2 & \dots & e^{\iota_2,\iota^{\prime}_2}_{N}  
	\end{pmatrix}_{(N+1)\times (N+1)}$$
\end{thm}
\begin{proof} 
	By utilizing the matrices $\mathbb{H}_{M0},\mathbb{H}_{M1}$, and $\mathbb{H}_0$ in equation \eqref{e5}, we can derive a set of nonlinear equations as follows:\\
	
	for 	$ \iota_{2}=0^*,Q^*,0,1,2, \cdots, S , \iota^{\prime}_{2}=0^* $
	\begin{eqnarray}\label{d8}
		\sum_{\substack{j=0^*,Q^*}}\mathfrak{A}^{(\iota_{2},j)}\mathfrak{A}^{(j,\iota^{\prime}_{2})}\mathbb{W}_{\iota^{\prime}_{2}}+\sum^S_{\substack{j=0}}\mathfrak{A}^{(\iota_{2},j)}\mathfrak{A}^{(j,\iota^{\prime}_{2})}\mathbb{W}_{\iota^{\prime}_{2}}+\mathfrak{A}^{(\iota_{2},\iota^{\prime}_{2})}\mathbb{G}_{\iota^{\prime}_{2}}+ \mathfrak{A}^{(\iota_{2},1)}\mathbb{J}_{\iota^{\prime}_{2}}+\delta_{\iota_{2}\iota^{\prime}_{2}}\mathbb{B}_{\iota^{\prime}_{2}} =  \textbf{0}~~~~
	\end{eqnarray}
	for 	$ \iota_{2}=0^*,Q^*,0,1,2, \cdots, S , \iota^{\prime}_{2}=Q^* $
	\begin{eqnarray}\label{d9}
		\sum_{\substack{j=0^*,Q^*}}\mathfrak{A}^{(\iota_{2},j)}\mathfrak{A}^{(j,\iota^{\prime}_{2})}\mathbb{W}_{\iota^{\prime}_{2}}+\sum^S_{\substack{j=0}}\mathfrak{A}^{(\iota_{2},j)}\mathfrak{A}^{(j,\iota^{\prime}_{2})}\mathbb{W}_{\iota^{\prime}_{2}}+\mathfrak{A}^{(\iota_{2},\iota^{\prime}_{2})}\mathbb{G}_{\iota^{\prime}_{2}}+ \mathfrak{A}^{(\iota_{2},0^*)}\mathbb{J}_{\iota^{\prime}_{2}}+\delta_{\iota_{2}\iota^{\prime}_{2}}\mathbb{B}_{\iota^{\prime}_{2}} =  \textbf{0}~~~~
	\end{eqnarray}
	for 	$ \iota_{2}=0^*,Q^*,0,1,2, \cdots, S , \iota^{\prime}_{2}=0 $
	\begin{eqnarray}\label{d10}\nonumber
		\sum_{\substack{j=0^*,Q^*}}\mathfrak{A}^{(\iota_{2},j)}\mathfrak{A}^{(j,\iota^{\prime}_{2})}\mathbb{W}_{\iota^{\prime}_{2}}+\sum^S_{\substack{j=0}}\mathfrak{A}^{(\iota_{2},j)}\mathfrak{A}^{(j,\iota^{\prime}_{2})}\mathbb{W}_{\iota^{\prime}_{2}}+\mathfrak{A}^{(\iota_{2},\iota^{\prime}_{2})}\mathbb{G}_{\iota^{\prime}_{2}}+ \mathfrak{A}^{(\iota_{2},0^*)}\mathbb{J}_{\iota^{\prime}_{2}}\\+\mathfrak{A}^{(\iota_{2},\iota^{\prime}_{2}+1)}\mathbb{F}_{\iota^{\prime}_{2}+1}+\delta_{\iota_{2}\iota^{\prime}_{2}}\mathbb{B}_{\iota^{\prime}_{2}} =  \textbf{0}
	\end{eqnarray}
	for 	$ \iota_{2}=0^*,Q^*,0,1,2, \cdots, S , \iota^{\prime}_{2}=1,2,\cdots,c$
	\begin{eqnarray}\label{d11}
		\hspace*{-1cm}
		\sum_{\substack{j=0^*,Q^*}}\mathfrak{A}^{(\iota_{2},j)}\mathfrak{A}^{(j,\iota^{\prime}_{2})}\mathbb{W}_{\iota^{\prime}_{2}}+\sum^S_{\substack{j=0}}\mathfrak{A}^{(\iota_{2},j)}\mathfrak{A}^{(j,\iota^{\prime}_{2})}\mathbb{W}_{\iota^{\prime}_{2}}+\mathfrak{A}^{(\iota_{2},\iota^{\prime}_{2})}\mathbb{G}_{\iota^{\prime}_{2}}+\mathfrak{A}^{(\iota_{2},\iota^{\prime}_{2}+1)}\mathbb{F}_{\iota^{\prime}_{2}+1}+\delta_{\iota_{2}\iota^{\prime}_{2}}\mathbb{B}_{\iota^{\prime}_{2}} =  \textbf{0}~~~~
	\end{eqnarray}
	for 	$ \iota_{2}=0^*,Q^*,0,1,2, \cdots, S , \iota^{\prime}_{2}=c+1,c+2,\cdots,Q-1$
	\begin{eqnarray}\label{d12}
		\sum_{\substack{j=0^*,Q^*}}\mathfrak{A}^{(\iota_{2},j)}\mathfrak{A}^{(j,\iota^{\prime}_{2})}\mathbb{W}_{c}+\sum^S_{\substack{j=0}}\mathfrak{A}^{(\iota_{2},j)}\mathfrak{A}^{(j,\iota^{\prime}_{2})}\mathbb{W}_{c}+\mathfrak{A}^{(\iota_{2},\iota^{\prime}_{2})}\mathbb{G}_{c}+\mathfrak{A}^{(\iota_{2},\iota^{\prime}_{2}+1)}\mathbb{F}_{c+1}+\delta_{\iota_{2}\iota^{\prime}_{2}}\mathbb{B}_{c} =  \textbf{0}~~~~
	\end{eqnarray}
	for 	$ \iota_{2}=0^*,Q^*,0,1,2, \cdots, S , \iota^{\prime}_{2}=Q $
	\begin{eqnarray}\label{d13}\nonumber
		\sum_{\substack{j=0^*,Q^*}}\mathfrak{A}^{(\iota_{2},j)}\mathfrak{A}^{(j,\iota^{\prime}_{2})}\mathbb{W}_{c}+\sum^S_{\substack{j=0}}\mathfrak{A}^{(\iota_{2},j)}\mathfrak{A}^{(j,\iota^{\prime}_{2})}\mathbb{W}_{c}+\mathfrak{A}^{(\iota_{2},\iota^{\prime}_{2})}\mathbb{G}_{c}+ \mathfrak{A}^{(\iota_{2},Q^*)}\mathbb{J}_{\iota^{\prime}_{2}}+\\\mathfrak{A}^{(\iota_{2},Q-\iota^{\prime}_{2})}\mathbb{L}_{Q-\iota^{\prime}_{2}}+\mathfrak{A}^{(\iota_{2},\iota^{\prime}_{2}+1)}\mathbb{F}_{c+1}+\delta_{\iota_{2}\iota^{\prime}_{2}}\mathbb{B}_{c} =  \textbf{0}
	\end{eqnarray}
	for 	$ \iota_{2}=0^*,Q^*,0,1,2, \cdots, S , \iota^{\prime}_{2}=Q+1,Q+2,\cdots,Q+c $
	\begin{eqnarray}\label{d14}\nonumber
		\sum_{\substack{j=0^*,Q^*}}\mathfrak{A}^{(\iota_{2},j)}\mathfrak{A}^{(j,\iota^{\prime}_{2})}\mathbb{W}_{c}+\sum^S_{\substack{j=0}}\mathfrak{A}^{(\iota_{2},j)}\mathfrak{A}^{(j,\iota^{\prime}_{2})}\mathbb{W}_{c}+\mathfrak{A}^{(\iota_{2},\iota^{\prime}_{2})}\mathbb{G}_{c}+ \mathfrak{A}^{(\iota_{2},Q-\iota^{\prime}_{2})}\mathbb{L}_{Q-\iota^{\prime}_{2}}+\\\mathfrak{A}^{(\iota_{2},\iota^{\prime}_{2}+1)}\mathbb{F}_{c+1}+\delta_{\iota_{2}\iota^{\prime}_{2}}\mathbb{B}_{c} =  \textbf{0}
	\end{eqnarray}
	for 	$ \iota_{2}=0^*,Q^*,0,1,2, \cdots, S , \iota^{\prime}_{2}=Q+c+1,Q+c+2,\cdots,S $
	\begin{eqnarray}\label{d15}\nonumber
		\sum_{\substack{j=0^*,Q^*}}\mathfrak{A}^{(\iota_{2},j)}\mathfrak{A}^{(j,\iota^{\prime}_{2})}\mathbb{W}_{c}+\sum^S_{\substack{j=0}}\mathfrak{A}^{(\iota_{2},j)}\mathfrak{A}^{(j,\iota^{\prime}_{2})}\mathbb{W}_{c}+\mathfrak{A}^{(\iota_{2},\iota^{\prime}_{2})}\mathbb{G}_{c}+ \mathfrak{A}^{(\iota_{2},Q-\iota^{\prime}_{2})}\mathbb{L}_{c}+\\\mathfrak{A}^{(\iota_{2},\iota^{\prime}_{2}+1)}\mathbb{F}_{c+1}+\delta_{\iota_{2}\iota^{\prime}_{2}}\mathbb{B}_{c} =  \textbf{0}
	\end{eqnarray}
	solving equations \eqref{d8} to \eqref{d15} by Gauss-Seidel iterative method, we will obtain the $ \mathbb{R} $ matrix.
\end{proof}
\begin{thm}
	The stationary probability vector $	\Phi^{(\iota_1)}$ of the Markov chain can be determined by
	$$\Phi^{(\iota_1)}=\begin{cases}
		\Phi^{(0)}\Omega_{\iota_1} & \hbox { if }  \iota_1= 0,1,2, \cdots,M \\
		\Phi^{(0)}\Omega_{M} \mathbb{R}^{(\iota_1-M)}  &  \hbox { if } \iota_1>M
	\end{cases}$$ 
	Where $\mathbb{R}$ is obtained from the equation \eqref{e5}, it can be obtained from Theorem \eqref{t10},  $$\Omega_{ \iota_1}=\begin{cases}
		I  &\hbox { if }  \iota_1 = 0 \\
		\prod\limits_{j=0}^{ \iota_1} \mathbb{H}_{0}K_{j} & \hbox { if }  \iota_1=1,2, \cdots, M	 
	\end{cases}$$
	$$K_{j}=\begin{cases}
		[-{(\mathbb{H}_{0}K_{j+1}\mathbb{H}_{j+2,0}+\mathbb{H}_{j+1,1})}]^{-1} & \hbox { if } j=1,2, \cdots,M-1\\
		[-(\mathbb{H}_{j,1}+\mathbb{R}\mathbb{H}_{j,0})]^{-1} &  \hbox { if } j=M
	\end{cases}$$
	and $ \Phi^{(0)}$ is obtained by solving  the system of equations
	\begin{eqnarray*}
		\Phi^{(0)}(\mathbb{H}_{0,1}+	\Omega_{1}\mathbb{H}_{1,0})&=&\textbf{0} \\
		\sum\limits_{\iota_1=0}^{\infty}\Phi^{(\iota_{1})}\textbf{e}& = &1.
	\end{eqnarray*}
\end{thm}
\begin{proof}
	let $ \Phi=(\Phi^{(0)},\Phi^{(1)},\Phi^{(2)},\cdots )$ be a probability vector which satisfies
	\begin{eqnarray}
		\Phi \hat{H}=\textbf{0} ~~~~\text{and}~~~~ \Phi \textbf{e}=1
	\end{eqnarray}
	Assume that 
	\begin{eqnarray}
		\label{eee}		\Phi^{(\iota_{1})}=	\Phi^{(M)}\mathbb{R}^{(\iota_{1}-M)} & \text{if}~~ \iota_{1}=M, M+1,\cdots 
	\end{eqnarray}
	Solving the $\Phi \hat{H}=\textbf{0}$, we get the following system of equations
	\begin{align}
		\Phi^{(0)}\mathbb{H}_{0,1}+	\Phi^{(1)}\mathbb{H}_{1,0}&=\textbf{0} \label{e11}\\
		\Phi^{(i-1)}\mathbb{H}_{0}+	\Phi^{(i)}\mathbb{H}_{i,1}+	\Phi^{(i+1)}\mathbb{H}_{i+1,0}&=\textbf{0} ~~\forall ~~~i=1,2, \cdots, M-1 \label{e12}\\
		\Phi^{(M-1)}\mathbb{H}_{0}+	\Phi^{(M)}(\mathbb{H}_{M,1}+\mathbb{R} \mathbb{H}_{M,0})&=\textbf{0} \label{e13}\\
		\text{and }~~\left[\sum \limits _{\iota_{1}=0}^{M-1}\Phi^{(\iota_{1})}+	\Phi^{(M)}(1-\mathbb{R})^{-1}\right] \textbf{e}&=1  \label{e14} 	
	\end{align}
	To solve the system of equations, we can use the following steps:\\
	\begin{itemize}
		\item[\textbf{Step~1:}] ~~~~Express $\Phi^{(M)}$ in terms of $\Phi^{(M-1)}$ using equation \eqref{e13}.
		$$\Phi^{(M)}=	\Phi^{(M-1)}\mathbb{H}_{0}K_{M}, 
		\text{ where } K_{M}=[-(\mathbb{H}_{M,1}+\mathbb{R} \mathbb{H}_{M,0})]^{-1}.$$ 
		\item[\textbf{Step~2:}]~~~~Express each $\Phi^{(i)}$ in terms of $\Phi^{(i-1)}$  for $i=1,2,\dots,M-1$ using \eqref{e12}.\\
		Substitute $i=M-1$ in \eqref{e12}, we get $$	\Phi^{(M-1)}=\Phi^{(M-2)}\mathbb{H}_{0}K_{M-1}
		\text{ where } K_{M-1}=[-(\mathbb{H}_{M-1,1}+\mathbb{H}_{0}K_{M}\mathbb{H}_{M,0})]^{-1}.$$
		Substitute $i=M-2$ in \eqref{e12}, we get $$	\Phi^{(M-2)}=	\Phi^{(M-3)}\mathbb{H}_{0}K_{M-2}, 
		\text{ where } K_{M-2}=[-(\mathbb{H}_{M-2,1}+\mathbb{H}_{0}K_{M-1}\mathbb{H}_{M-1,0})]^{-1}.$$ 
		In general solving \eqref{e12} recursively, 
		\begin{align}
			\Phi^{(\iota_{1})}=	\Phi^{(\iota_{1}-1)}\mathbb{H}_{0}K_{\iota_{1}}~ \forall ~\iota_{1}=1,2, \cdots, M \label{e15}
		\end{align}
		$K_{\iota_{1}}=\begin{cases}
			[-(\mathbb{H}_{\iota_{1},1}+\mathbb{R} \mathbb{H}_{\iota_{1},0})]^{-1}  & \iota_{1}=M\\
			[-(\mathbb{H}_{\iota_{1},1}+\mathbb{H}_{0}K_{\iota_{1}+1}\mathbb{H}_{\iota_{1},0})]^{-1} &   \iota_{1}=1, \cdots, M-1.\\
		\end{cases}$
		\item[\textbf{Step~3:}]~~~~ Express each $\Phi^{(i)}$ for $i=1,2,\dots M$ in terms of $\Phi^{(0)}$ using\eqref{e15}, we have
		\begin{align}
			\Phi^{(\iota_1)}=	\Phi^{(0)} \Omega_{\iota_1}~\forall~\iota_1=0,1,2, \cdots, M \label{e16}
		\end{align}
		where
		\begin{align*}
			\Omega_{\iota_1}=\begin{cases}
				I  &\hbox { if }  \iota_1 = 0 \\
				\prod\limits_{j=0}^{ \iota_1} \mathbb{H}_{0}K_{j} & \hbox { if }  \iota_1=1,2, \cdots, M.	 
			\end{cases}
		\end{align*}
		\item[\textbf{Step~4:}]~~~~ Substitute \eqref{e16}  in \eqref{e14}, we get
		\begin{eqnarray}\label{e34}
			\Phi^{(0)}[\sum\limits_{k=0}^{M-1}\Omega_{k}+\Omega_{M}(1-\mathbb{R})^{-1}] \textbf{e}=1.
		\end{eqnarray}
		Using the equations \eqref{e11}  and  \eqref{e34}, we obtain $\Phi^{(0)}$.
		
	\end{itemize}
\end{proof}	
\section{System Performance Metrics}\label{cha}
Using the steady-state probability vector, various system performance metrics are defined in this section as follows:
\begin{enumerate}
	\item 	The mean inventory level ($ \curlywedge_1$) of the MQIS in steady state is defined by using the positive inventory and its corresponding vector $\Phi$.
	\begin{eqnarray*} \curlywedge_1=\sum\limits_{\iota_1=0}^{\infty}\sum\limits_{\iota_6=0}^{N} Q	\Phi^{(\iota_1,Q,c,0,0,\iota_6)}+\sum\limits_{\iota_1=0}^{\infty}\sum\limits_{\iota_2=1}^{S} \iota_2	\Phi^{(\iota_1,\iota_2)}\bf{e}
	\end{eqnarray*}
	
	\item Under the $(s,Q)$ policy, a reorder is immediately placed when the inventory level drops to $s$ at the end of the service. Using the vector $\Phi$, we can figure out the steady-state mean reorder rate ($ \curlywedge_2$) of the MQIS:
	\begin{eqnarray*}	 \curlywedge_2=\sum\limits_{\iota_1=0}^{\infty}\sum\limits_{\iota_3=0}^{c-1} \sum\limits_{\iota_4=1}^{c-\iota_3} \iota_4 \mu 	\Phi^{(\iota_1,s+1,\iota_3,\iota_4,c-(\iota_3+\iota_4),\iota_4)}+\sum\limits_{\iota_1=0}^{\infty}\sum\limits_{\iota_3=0}^{c-1} \sum\limits_{\iota_6=c-\iota_{3}+1}^{N} 	(c-\iota_{3})\mu \Phi^{(\iota_1,s+1,\iota_3,c-\iota_3,0,\iota_6)}
	\end{eqnarray*}
	\item The expected number of customers in orbit is denoted by $\curlywedge_3$. When arriving customers discover that the waiting hall is full, they enter the orbit with probability $p$. At any given time, the average number of orbital customers present is
	\begin{eqnarray*}
		\curlywedge_3=\sum\limits_{\iota_1=0}^{M-1} \iota_1	\Phi^{(\iota_1)}{\bf e}+M\left(\dfrac{1}{I-\mathbb{R}}\right) \Omega_M \Phi^{(0)}{\bf e}
	\end{eqnarray*}
	
	\item The expected number of customers entering an orbit is denoted by $\curlywedge_4$. When arriving customers find that the waiting hall has reached its maximum level, they enter the orbit under the Bernoulli schedule with probability $p$. At any given time, the average number of customers entering an orbit is
	\begin{eqnarray*}
		\curlywedge_4=\sum\limits_{\iota_1=0}^{\infty}p\lambda 	\Phi^{(\iota_1,0,c,0,0,N)}+
		\sum\limits_{\iota_1=0}^{\infty}p\lambda 	\Phi^{(\iota_1,Q,c,0,0,N)}+\sum\limits_{\iota_1=0}^{\infty}\sum\limits_{\iota_2=0}^{c-1} \sum\limits_{\iota_3=0}^{c-(\iota_2+1)} p\lambda	\Phi^{(\iota_1,\iota_2,\iota_3,\iota_2,c-(\iota_3+\iota_2),N)}		\\+\sum\limits_{\iota_1=0}^{\infty}\sum\limits_{\iota_2=1}^{c-1} \sum\limits_{\iota_3=c-\iota_2}^{c-1} p\lambda	\Phi^{(\iota_1,\iota_2,\iota_3,c-\iota_3,0,N)}+			\sum\limits_{\iota_1=0}^{\infty}\sum\limits_{\iota_2=c}^{S} \sum\limits_{\iota_3=0}^{c-1} p\lambda	\Phi^{(\iota_1,\iota_2,\iota_3,c-\iota_3,0,N)}
	\end{eqnarray*}
	
	\item Let $\curlywedge_5$ represent the mean waiting time of the orbital customer and is defined by
	$$\curlywedge_5=\dfrac{\curlywedge_3}{\curlywedge_4}$$
	
	\item At least one customer has to wait in the waiting room in order to get the mean customer in the queue in the steady-state ($ \curlywedge_6 $) of the MQIS, which is defined by using the vector $ \Phi $, which is given by 
	\begin{adjustwidth}{-1cm}{1cm}
		\begin{align*}
			\curlywedge_6&=\sum\limits_{\iota_1=0}^{\infty} 
			\sum\limits_{\iota_6=1}^{N}\iota_6	 
			\Phi^{(\iota_1,0,c,0,0,\iota_6)}+
			\sum\limits_{\iota_1=0}^{\infty}  \sum\limits_{\iota_6=1}^{N}	\iota_6 \Phi^{(\iota_1,Q,c,0,0,\iota_6)}\\  &+\sum\limits_{\iota_1=0}^{\infty}\sum\limits_{\iota_2=0}^{c-1} \sum\limits_{\iota_3=0}^{c-(\iota_2+1)}\sum\limits_{\iota_6=c-\iota_3}^{N}  \iota_6	\Phi^{(\iota_1,\iota_2,\iota_3,\iota_2,c-(\iota_3+\iota_6),\iota_6)}+	\sum\limits_{\iota_1=0}^{\infty}\sum\limits_{\iota_2=1}^{c-1} \sum\limits_{\iota_3=c-\iota_2}^{c-1} \sum\limits_{\iota_6=1}^{c-\iota_{3}}  \iota_6	\Phi^{(\iota_1,\iota_2,\iota_3,\iota_6,c-(\iota_3+\iota_6),\iota_6)}	\\&+\sum\limits_{\iota_1=0}^{\infty}\sum\limits_{\iota_2=0}^{c-1} \sum\limits_{\iota_3=c-\iota_2}^{c-1}\sum\limits_{\iota_6=c-\iota_3+1}^{N}  \iota_6	\Phi^{(\iota_1,\iota_2,\iota_3,c-\iota_3,0,\iota_6)}+
			\sum\limits_{\iota_1=0}^{\infty}\sum\limits_{\iota_2=c}^{S} \sum\limits_{\iota_3=0}^{c-1}\sum\limits_{\iota_6=1}^{c-\iota_{3}}    \iota_6 \Phi^{(\iota_1,\iota_2,\iota_3,\iota_6,c-(\iota_3+\iota_6),\iota_6)}	\\&		+\sum\limits_{\iota_1=0}^{\infty}\sum\limits_{\iota_2=c}^{S} \sum\limits_{\iota_3=0}^{c-1}  \sum\limits_{\iota_6=c-\iota_3+1}^{N}  \iota_6	\Phi^{(\iota_1,\iota_2,\iota_3,c-\iota_3,0,\iota_6)}
		\end{align*}
	\end{adjustwidth}
	\item The symbol $\curlywedge_7$ represents the expected number of customers entering the waiting hall and is defined by
	\begin{align*}
		&\curlywedge_7=\sum\limits_{\iota_1=0}^{\infty} \sum\limits_{\iota_6=0}^{N-1}\lambda	\Phi^{(\iota_1,0,c,0,0,\iota_6)}+
		\sum\limits_{\iota_1=0}^{\infty}  \sum\limits_{\iota_6=0}^{N-1}	\lambda \Phi^{(\iota_1,Q,c,0,0,\iota_6)}\\	&+\sum\limits_{\iota_1=0}^{\infty}\sum\limits_{\iota_2=0}^{c-1} \sum\limits_{\iota_3=0}^{c-(\iota_2+1)}\sum\limits_{\iota_6=c-\iota_3}^{N-1}  \lambda	\Phi^{(\iota_1,\iota_2,\iota_3,\iota_2,c-(\iota_3+\iota_6),\iota_6)}+	\sum\limits_{\iota_1=0}^{\infty}\sum\limits_{\iota_2=1}^{c-1} \sum\limits_{\iota_3=c-\iota_2}^{c-1} \sum\limits_{\iota_6=0}^{c-\iota_{3}} \lambda	\Phi^{(\iota_1,\iota_2,\iota_3,\iota_6,c-(\iota_3+\iota_6),\iota_6)}\\&+	\sum\limits_{\iota_1=0}^{\infty}\sum\limits_{\iota_2=0}^{c-1} \sum\limits_{\iota_3=c-\iota_2}^{c-1}\sum\limits_{\iota_6=c-\iota_3+1}^{N-1} \lambda	\Phi^{(\iota_1,\iota_2,\iota_3,c-\iota_3,0,\iota_6)}+	\sum\limits_{\iota_1=0}^{\infty}\sum\limits_{\iota_2=c}^{S} \sum\limits_{\iota_3=0}^{c-1}\sum\limits_{\iota_6=0}^{c-\iota_{3}}   \lambda \Phi^{(\iota_1,\iota_2,\iota_3,\iota_6,c-(\iota_3+\iota_6),\iota_6)}	\\&	+\sum\limits_{\iota_1=0}^{\infty}\sum\limits_{\iota_2=c}^{S} \sum\limits_{\iota_3=0}^{c-1}  \sum\limits_{\iota_6=c-\iota_3+1}^{N-1}  \lambda	\Phi^{(\iota_1,\iota_2,\iota_3,c-\iota_3,0,\iota_6)}
	\end{align*}
	\item Let $\curlywedge_8$ represent the average waiting time of a customer in a waiting room.
	$$\curlywedge_8=\dfrac{\curlywedge_6}{\curlywedge_7}$$
	\item Let's say that the number of people waiting in line has reached its limit. The customer has the option of becoming lost. Using the vector $\Phi$, here the expected primary customer is lost in the queue   ($\curlywedge_9$) of the steady-state MQIS is defined as:
	\begin{align*}
		\curlywedge_9&=\sum\limits_{\iota_1=0}^{\infty}p\lambda 	\Phi^{(\iota_1,0,c,0,0,N)}+
		\sum\limits_{\iota_1=0}^{\infty}(1-p)\lambda 	\Phi^{(\iota_1,Q,c,0,0,N)}\\&+	\sum\limits_{\iota_1=0}^{\infty}\sum\limits_{\iota_2=0}^{c-1} \sum\limits_{\iota_3=0}^{c-(\iota_2+1)} (1-p)\lambda	\Phi^{(\iota_1,\iota_2,\iota_3,\iota_2,c-(\iota_3+\iota_2),N)}	+\\&	\sum\limits_{\iota_1=0}^{\infty}\sum\limits_{\iota_2=1}^{c-1} \sum\limits_{\iota_3=c-\iota_2}^{c-1} (1-p)\lambda	\Phi^{(\iota_1,\iota_2,\iota_3,c-\iota_3,0,N)}+	\sum\limits_{\iota_1=0}^{\infty}\sum\limits_{\iota_2=c}^{S} \sum\limits_{\iota_3=0}^{c-1} (1-p)\lambda	\Phi^{(\iota_1,\iota_2,\iota_3,c-\iota_3,0,N)}
	\end{align*}
	\item The current stock level and the length of the queue are both on the positive side. In the steady state of the MQIS, the mean number of busy servers, denoted by the symbol $\curlywedge_{10}$, is defined by using the vector $\Phi$ as follows:
	\begin{adjustwidth}{-1cm}{1cm}
		\begin{align*}	\curlywedge_{10}&=\sum\limits_{\iota_1=0}^{\infty}\sum\limits_{\iota_2=1}^{c-1} \sum\limits_{\iota_3=0}^{c-(\iota_2+1)}\sum\limits_{\iota_6=c-\iota_3}^{N}  \iota_4	\Phi^{(\iota_1,\iota_2,\iota_3,\iota_2,c-(\iota_3+\iota_2),\iota_6)}+	\sum\limits_{\iota_1=0}^{\infty}\sum\limits_{\iota_2=1}^{c-1} \sum\limits_{\iota_3=c-\iota_2}^{c-1} \sum\limits_{\iota_4=1}^{c-\iota_{3}} \iota_4	\Phi^{(\iota_1,\iota_2,\iota_3,\iota_4,c-(\iota_3+\iota_4),\iota_4)}	\\&+	\sum\limits_{\iota_1=0}^{\infty}\sum\limits_{\iota_2=1}^{c-1} \sum\limits_{\iota_3=c-\iota_2}^{c-1}\sum\limits_{\iota_6=c-\iota_3+1}^{N} (c-\iota_3)	\Phi^{(\iota_1,\iota_2,\iota_3,c-\iota_3,0,\iota_6)}+	\sum\limits_{\iota_1=0}^{\infty}\sum\limits_{\iota_2=c}^{S} \sum\limits_{\iota_3=0}^{c-1}\sum\limits_{\iota_4=1}^{c-\iota_{3}}   \iota_4 \Phi^{(\iota_1,\iota_2,\iota_3,\iota_4,c-(\iota_3+\iota_4),\iota_4)}	\\&	+\sum\limits_{\iota_1=0}^{\infty}\sum\limits_{\iota_2=c}^{S} \sum\limits_{\iota_3=0}^{c-1}  \sum\limits_{\iota_6=c-\iota_3+1}^{N}   (c-\iota_3)	\Phi^{(\iota_1,\iota_2,\iota_3,c-\iota_3,0,\iota_6)}
		\end{align*}
	\end{adjustwidth}
	\item Whenever a server completes service and there are no sufficient customers in the queue or inventory, the server takes a vacation. Also, if the server returns from a vacation to satisfy the vacation policy, they immediately leave for another vacation; otherwise, they return to serve the customer. In the steady state of the MQIS, the mean number of servers on vacation, denoted by the symbol $\curlywedge_{11}$, is defined by using the vector $\Phi$ as follows:
	\begin{adjustwidth}{-1cm}{1cm}
		\begin{align*}
			\curlywedge_{11}=&=\sum\limits_{\iota_1=0}^{\infty} \sum\limits_{\iota_6=0}^{N}c	\Phi^{(\iota_1,0,c,0,0,\iota_6)}+
			\sum\limits_{\iota_1=0}^{\infty}  \sum\limits_{\iota_6=0}^{N}	c \Phi^{(\iota_1,Q,c,0,0,\iota_6)}\\	&+\sum\limits_{\iota_1=0}^{\infty}\sum\limits_{\iota_2=0}^{c-2} \sum\limits_{\iota_3=1}^{c-(\iota_2+1)}\sum\limits_{\iota_6=c-\iota_3}^{N} \iota_3	\Phi^{(\iota_1,\iota_2,\iota_3,\iota_2,c-(\iota_3+\iota_2),\iota_6)}+	\sum\limits_{\iota_1=0}^{\infty}\sum\limits_{\iota_2=1}^{c-1} \sum\limits_{\iota_3=c-\iota_2}^{c-1} \sum\limits_{\iota_4=0}^{c-\iota_{3}}  \iota_3	\Phi^{(\iota_1,\iota_2,\iota_3,\iota_4,c-(\iota_3+\iota_4),\iota_4)}	\\&+	\sum\limits_{\iota_1=0}^{\infty}\sum\limits_{\iota_2=1}^{c-1} \sum\limits_{\iota_3=c-\iota_2}^{c-1}\sum\limits_{\iota_6=c-\iota_3+1}^{N}  \iota_3	\Phi^{(\iota_1,\iota_2,\iota_3,c-\iota_3,0,\iota_6)}+	\sum\limits_{\iota_1=0}^{\infty}\sum\limits_{\iota_2=c}^{S} \sum\limits_{\iota_3=1}^{c-1}\sum\limits_{\iota_4=0}^{c-\iota_{3}}    \iota_3 \Phi^{(\iota_1,\iota_2,\iota_3,\iota_4,c-(\iota_3+\iota_4),\iota_4)}	\\&	+\sum\limits_{\iota_1=0}^{\infty}\sum\limits_{\iota_2=c}^{S} \sum\limits_{\iota_3=1}^{c-1}  \sum\limits_{\iota_6=c-\iota_3+1}^{N}  \iota_3	\Phi^{(\iota_1,\iota_2,\iota_3,c-\iota_3,0,\iota_6)}
		\end{align*}
	\end{adjustwidth}
	\item At the end of the vacation or service completion, if the system has at least one sufficient customer in the queue or one sufficient item in the inventory, the server moves to idle mode. Therefore, the mean number of idle servers $\curlywedge_{12}$ is defined by
	\begin{align*}
		&\curlywedge_{12}=c-(\curlywedge_{10}+\curlywedge_{11})
	\end{align*}
	\item The possibility of a retrial customer trying to enter the accommodation place is indicated by the notation $\curlywedge_{13}$ and is known as the overall rate of retrial. Using the mean value $\curlywedge_{3}$ and rate $\theta $ , the $\curlywedge_{13}$ of SQIS at steady-state is defined by
	\begin{align*}
		&\curlywedge_{13}=\theta \curlywedge_{3}
	\end{align*}
	\item The successful rate of retrial occurs when a retrial demand is able to enter the queue if its size is less than $N$. The mean number of successful retrials is denoted by $\curlywedge_{14}$ and it is defined as:
	\begin{adjustwidth}{-1cm}{1cm}
		\begin{align*}
			&\curlywedge_{14}=\sum\limits_{\iota_1=1}^{\infty} \sum\limits_{\iota_6=0}^{N-1}\iota_1\theta	\Phi^{(\iota_1,0,c,0,0,\iota_6)}+
			\sum\limits_{\iota_1=0}^{\infty}  \sum\limits_{\iota_6=0}^{N-1}	\iota_1\theta \Phi^{(\iota_1,Q,c,0,0,\iota_6)}\\	&+\sum\limits_{\iota_1=1}^{\infty}\sum\limits_{\iota_2=0}^{c-1} \sum\limits_{\iota_3=0}^{c-(\iota_2+1)}\sum\limits_{\iota_6=c-\iota_3}^{N-1}\iota_1\theta	\Phi^{(\iota_1,\iota_2,\iota_3,\iota_2,c-(\iota_3+\iota_6),\iota_6)}+	\sum\limits_{\iota_1=1}^{\infty}\sum\limits_{\iota_2=1}^{c-1} \sum\limits_{\iota_3=c-\iota_2}^{c-1} \sum\limits_{\iota_6=0}^{c-\iota_{3}} \iota_1\theta	\Phi^{(\iota_1,\iota_2,\iota_3,\iota_6,c-(\iota_3+\iota_6),\iota_6)}	\\&+	\sum\limits_{\iota_1=1}^{\infty}\sum\limits_{\iota_2=0}^{c-1} \sum\limits_{\iota_3=c-\iota_2}^{c-1}\sum\limits_{\iota_6=c-\iota_3+1}^{N-1} \iota_1\theta	\Phi^{(\iota_1,\iota_2,\iota_3,c-\iota_3,0,\iota_6)}+	\sum\limits_{\iota_1=1}^{\infty}\sum\limits_{\iota_2=c}^{S} \sum\limits_{\iota_3=0}^{c-1}\sum\limits_{\iota_6=0}^{c-\iota_{3}}   \iota_1\theta \Phi^{(\iota_1,\iota_2,\iota_3,\iota_6,c-(\iota_3+\iota_6),\iota_6)}	\\&	+\sum\limits_{\iota_1=1}^{\infty}\sum\limits_{\iota_2=c}^{S} \sum\limits_{\iota_3=0}^{c-1}  \sum\limits_{\iota_6=c-\iota_3+1}^{N-1}  \iota_1\theta	\Phi^{(\iota_1,\iota_2,\iota_3,c-\iota_3,0,\iota_6)}
		\end{align*}
	\end{adjustwidth}
	\item The mean number fraction of successful retrials $\curlywedge_{15}$ is defined as the ratio of the mean number of overall retrials to the mean number of successful retrials and is given by
	\begin{align*}
		&\curlywedge_{15}=\dfrac{\curlywedge_{14}}{\curlywedge_{13}}
	\end{align*}
	\item There are two possible scenarios where all servers will be in vacation mode: the first is when the current inventory level is empty, and the second is when the ordered quantity of $Q$ items has been received while the inventory level is still empty. In both situations, all servers will be on vacation. Consequently, the probability of all servers being on vacation can be calculated as follows:
	\begin{align*}	&\curlywedge_{16}=\sum\limits_{\iota_1=0}^{\infty}\sum\limits_{\iota_2=0,Q}  \sum\limits_{\iota_6=0}^{N} 	\Phi^{(\iota_1,\iota_2, c,0,0,\iota_6)}
	\end{align*}
\end{enumerate}

\section{Cost Analysis and Numerical Illustration}\label{cos}
The expected total cost (ETC) of the proposed model is given by
$$ETC= ch*\curlywedge_1 + cs*\curlywedge_2+co*\curlywedge_3+cw*\curlywedge_6+cl*\curlywedge_9.$$
\subsection{Parameter Analysis}
The cost and parameter values of the system will be utilized to investigate the six-dimensional stochastic multi-server queuing-inventory problem under consideration.  By varying the parameters, the characteristics of the total cost, customer loss, busy servers, a fraction of successful retrial rate, and waiting time of customers in the orbit and waiting hall are studied. As per the result obtained in the stability condition and normalizing property, the following parameters and costs of the Markov process are assumed: $S=32;s=10; ch =0.01; cs=3; co=1; cw=1.3;cl=0.01;c=3; Q=S-s; N=4; M=5;p=0.7;\theta=0.7;\lambda=2.5;\mu=5;\eta=2.7;\beta=1.5;$ for the analysis of the numerical discussions.
\begin{ex}
	In this example, the impact of the parameters $\lambda$, $\mu$, $\theta$, $\beta$, $\eta$, and $c$ on the total cost can be observed in Table \ref{tab1}. 
	\begin{enumerate}
		\item It has been noted that as the parameter value $\lambda$ increases, the total cost also increases. This is due to the fact that the arrival rate increases, resulting in a greater number of customers in the waiting hall. Therefore, increasing the parameter value $\lambda$ leads to an increase in the total cost.
		\item  As the $\theta$ increases, the total cost also increases. This is because when the retrial rate increases, more customers enter the waiting hall. As a result, the product of $cw$ and $\curlywedge_6$ is increased, leading to an increase in total cost. 
		\item It is evident that the total cost decreases as the value of $\mu$ increases. This is due to the fact that an increase in $\mu$ leads to a decrease in the average service time per customer, as well as a reduction in the size of the waiting room.
		\item From the result in Table \ref{tab1}, we observe that if the vacation completion rate increases, the total cost decreases along with it. As soon as vacation completion occurs, the customer wait in the queue also reduces quickly. Thus, it reduces the total cost.
		\item While increasing the $c$ in the system, the expected total cost of the system is also increased.             
		\item The expected total cost drops whenever the $\beta$ increases. This implies that the setup cost per order is decreasing, which results in a decrease in the overall cost. It also signifies that the average replenishment time is decreasing.
	\end{enumerate}
\newpage
	\textbf{Insights:} Every business's primary goal is to reduce total costs. In such a way, this example investigates the total cost of the change of various parameters and their characteristics. The parameter analysis provides insight to business people, allowing them to develop new strategies and ideas. 
	\vspace{.2 cm}
	{\footnotesize
		\setlength\LTleft{-1.4cm}
\begin{longtable}{|c|c|c|c|cccc|ccc|ccc|}
			\caption{Expected total cost under the influence of parameters}
			\label{tab1}\\
			\hline
			\multirow{2}{*}{$c$}  & \multirow{2}{*}{$\eta$}   & \multirow{2}{*}{$\theta$}   & \multirow{2}{*}{$\beta$} & \multicolumn{1}{c|}{$\lambda$} & \multicolumn{3}{c|}{2.4}                                            & \multicolumn{3}{c|}{2.5}                                                 & \multicolumn{3}{c|}{2.6}                                                 \\ \cline{5-14} 
			&                      &                      &                    & \multicolumn{1}{c|}{$\mu$} & \multicolumn{1}{c|}{4.9} & \multicolumn{1}{c|}{5}        & 5.1      & \multicolumn{1}{c|}{4.9}      & \multicolumn{1}{c|}{5}        & 5.1      & \multicolumn{1}{c|}{4.9}      & \multicolumn{1}{c|}{5}        & 5.1      \\ \hline
			\endfirsthead
			\multicolumn{14}{c}%
			{{\bfseries Table \thetable\ continued from previous page}} \\
			\hline
			\multirow{2}{*}{$ c $}  & \multirow{2}{*}{$ \eta $}   & \multirow{2}{*}{$ \theta $}   & \multirow{2}{*}{$ \beta $} & \multicolumn{1}{c|}{$ \lambda $} & \multicolumn{3}{c|}{2.4}                                            & \multicolumn{3}{c|}{2.5}                                                 & \multicolumn{3}{c|}{2.6}                                                 \\ \cline{5-14} 
			&                      &                      &                    & \multicolumn{1}{c|}{$ \mu $} & \multicolumn{1}{c|}{4.9} & \multicolumn{1}{c|}{5}        & 5.1      & \multicolumn{1}{c|}{4.9}      & \multicolumn{1}{c|}{5}        & 5.1      & \multicolumn{1}{c|}{4.9}      & \multicolumn{1}{c|}{5}        & 5.1      \\ \hline
			\endhead
			\multirow{27}{*}{2} & \multirow{9}{*}{2.6} & \multirow{3}{*}{0.6} & 1                  & \multicolumn{2}{c|}{3.606676}                     & \multicolumn{1}{c|}{3.593941} & 3.581558 & \multicolumn{1}{c|}{3.614909} & \multicolumn{1}{c|}{3.601799} & 3.589070 & \multicolumn{1}{c|}{3.625045} & \multicolumn{1}{c|}{3.611544} & 3.598455 \\ \cline{4-14} 
			&                      &                      & 1.5                & \multicolumn{2}{c|}{3.477352}                     & \multicolumn{1}{c|}{3.462076} & 3.447430 & \multicolumn{1}{c|}{3.492117} & \multicolumn{1}{c|}{3.476244} & 3.461000 & \multicolumn{1}{c|}{3.507800} & \multicolumn{1}{c|}{3.491344} & 3.475522 \\ \cline{4-14} 
			&                      &                      & 2                  & \multicolumn{2}{c|}{3.451354}                     & \multicolumn{1}{c|}{3.432587} & 3.414337 & \multicolumn{1}{c|}{3.469717} & \multicolumn{1}{c|}{3.450710} & 3.432218 & \multicolumn{1}{c|}{3.488554} & \multicolumn{1}{c|}{3.469285} & 3.450532 \\ \cline{3-14} 
			&                      & \multirow{3}{*}{0.7} & 1                  & \multicolumn{2}{c|}{3.779718}                     & \multicolumn{1}{c|}{3.767147} & 3.754890 & \multicolumn{1}{c|}{3.785414} & \multicolumn{1}{c|}{3.772513} & 3.759951 & \multicolumn{1}{c|}{3.792931} & \multicolumn{1}{c|}{3.779690} & 3.766815 \\ \cline{4-14} 
			&                      &                      & 1.5                & \multicolumn{2}{c|}{3.666558}                     & \multicolumn{1}{c|}{3.652014} & 3.638002 & \multicolumn{1}{c|}{3.678276} & \multicolumn{1}{c|}{3.663161} & 3.648577 & \multicolumn{1}{c|}{3.690855} & \multicolumn{1}{c|}{3.675189} & 3.660055 \\ \cline{4-14} 
			&                      &                      & 2                  & \multicolumn{2}{c|}{3.645263}                     & \multicolumn{1}{c|}{3.627694} & 3.610569 & \multicolumn{1}{c|}{3.660726} & \multicolumn{1}{c|}{3.642942} & 3.625598 & \multicolumn{1}{c|}{3.676603} & \multicolumn{1}{c|}{3.658586} & 3.641005 \\
\hline
		\end{longtable}
	}
\end{ex}
\begin{ex}
	In this example, we study the impact of the expected customer loss rate by varying the parameters $\lambda,~\mu,~\theta,~\beta,~\eta,$ and $c$ in Table \ref{tab2}.
	\begin{enumerate}
		\item  Whenever the $\lambda$, and $\theta$ increase, the mean number of customers lost also increases due to waiting hall overflow.    
		\item When $\mu$ increases, the expected customer loss rate decreases. As the service time per customer decreases, the number of customers in the waiting hall will decrease simultaneously. So that a new customer who arrives can easily join the queue. 
		\item  When the rate $\eta$ increases, the value of $\curlywedge_9$ decreases. This is because when the server finishes his vacation, he immediately starts providing service. Consequently, customers in the waiting hall do not have to wait for an extended period, and the waiting area remains uncluttered. As a result, the average number of lost customers is effectively managed.
		\item When $\beta$ is increased, the existing stock level increases monotonically. It ensures that the system has no insufficient stocks in the inventory. Therefore, the customer in the waiting hall does not need to wait due to the shortage of inventory, and so the expected customer loss rate will be reduced.    
		\item  As we stated for the parameter $\mu$, increasing $c$ reduces the number of customers in the waiting hall quickly. So the overflow in the waiting room is then reduced, and the expected loss rate decreases automatically.   
	\end{enumerate} 
	
	\textbf{Insights:}  Customers are usually a vital part of any business. In business, the profit or growth is determined by the arrival rate of the customer and his purchase rate. When the arrival rate of a customer decreases, the sales process and profit are reduced. So a good administration always tries to reduce the customer loss rate of the business. In such a way, the above example gives the customer loss rate-related results to the readers.
\end{ex}
\begin{ex}
	In this example, we explained the impact of the parameters $\lambda,~\theta,~\mu,~c,~\beta$ and $\eta$ on the mean number of servers on vacation and the mean number of busy servers, as shown in figures \ref{fig:2}-\ref{fig:7}.
	\begin{enumerate}
		\item As shown in Figure \ref{fig:2}, as the rate $\lambda$ increases, the mean number of servers on vacation decreases. Because there exists a sufficient number of customers in the queue, vacation continuation is not allowed. As the rate $\mu$ increases, so does $\curlywedge_{11}$. Because $\mu$ increases, the service time per customer is decreased. 
		\item In figure \ref{fig:3}, if the change in $\theta$ is small, then the change in $\curlywedge_{11}$ is also small with an indirect variation. For the fixed waiting hall size, when $c=2$ and $c=3$, the change in $\curlywedge_{11}$ is small but when $c=4$, it is to be noted that the change in $\curlywedge_{11}$ is high.  
		\item As depicted in figure \ref{fig:4}, mean number of servers are in vacation reduces when $\beta$ increases. If the replenishment has been done quickly, then the server cannot be allowed to take a vacation. By increasing $\eta$, the amount of time a server spends on vacation is reduced. Therefore, the server must frequently return to the system, which confirms the decrease in the number of  servers on vacation.
		\item Figure \ref{fig:5} shows that the rate $\mu$ decreases the mean number of busy servers if it increases. It helps to decrease the number of customers in the queue. As the number of customers in the queue shortly decreases, the number of busy servers will also be reduced. The $\curlywedge _{10}$ will be increased as  $\lambda$ increases. Because of the increase in arrivals, all ideal servers will be activated instantly.
		\item If the parameters $c$ and $\theta$ increase, then the mean number of busy servers increases simultaneously since the arrival of a retrial customer enters the queue and activates each server. Once the number of servers is increased, he will immediately join the system to provide service to the customers. This is shown in figure \ref{fig:6}.
		\item In figure \ref{fig:7} $\beta$  and $\eta$  are directly proportional to $\curlywedge_{10}$. If the change in $\eta$ is small then the $\curlywedge_{10}$ is also small.         
	\end{enumerate}
	\textbf{Insights:}
	Analyzing the mean number of servers on vacation and the mean number of busy servers can provide valuable insights for business professionals. By examining the impact of various parameters on $\curlywedge_{10}$ and $\curlywedge_{11}$, readers can develop new strategies and ideas to reduce the average number of servers on vacation and increase the average number of busy servers in their businesses.
		\begin{figure}[]
		\begin{adjustwidth}{-.6cm}{-.4cm}
			\begin{minipage}[b]{0.54\linewidth}
				\includegraphics[width=1\linewidth]{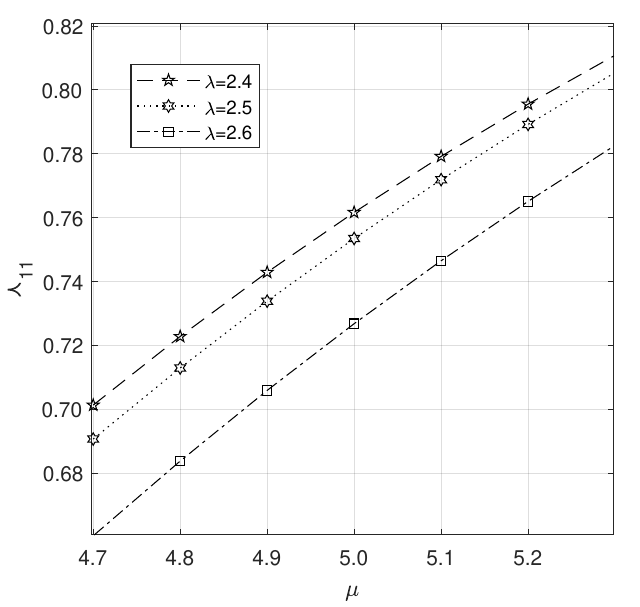} 
				\caption{Mean number of server is on vacation vs $\lambda$ and $\mu$} 
				\label{fig:2}
			\end{minipage}
			\begin{minipage}[b]{0.56\linewidth}
					\includegraphics[width=1.02\linewidth]{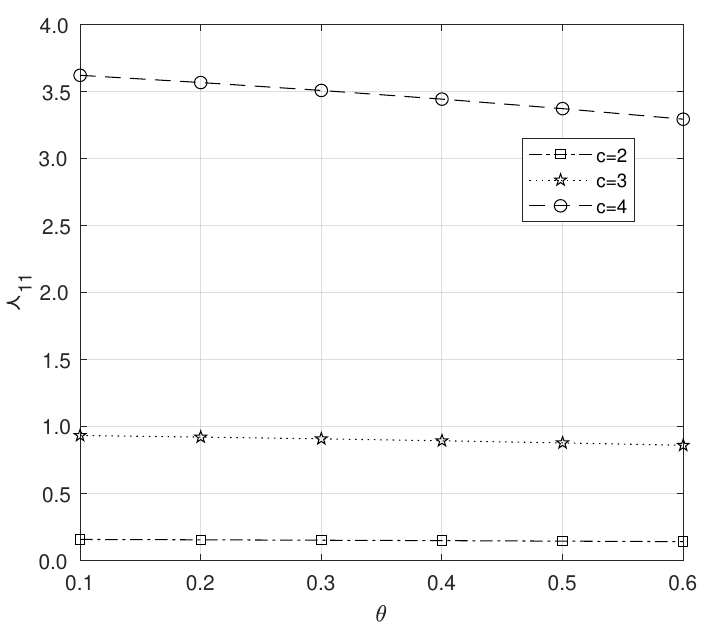} 
				\caption{Mean number of server is on vacation vs $c$ and $\theta$ } 
				\label{fig:3}
			\end{minipage} 
			\begin{minipage}[h]{0.55\linewidth}
					\includegraphics[width=1\linewidth]{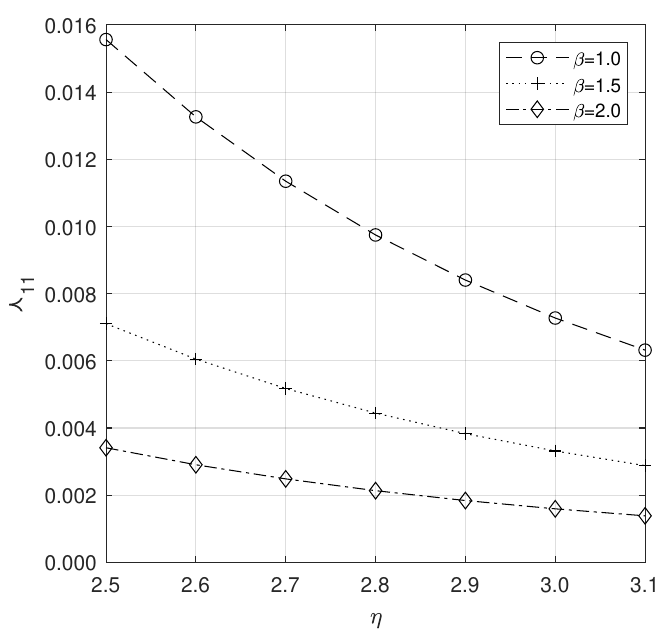} 
				\caption{Mean number of server is on vacation vs $\beta$ and $\eta$} 
				\label{fig:4}
			\end{minipage} 
			\begin{minipage}[h]{0.55\linewidth}
				\includegraphics[width=1\linewidth]{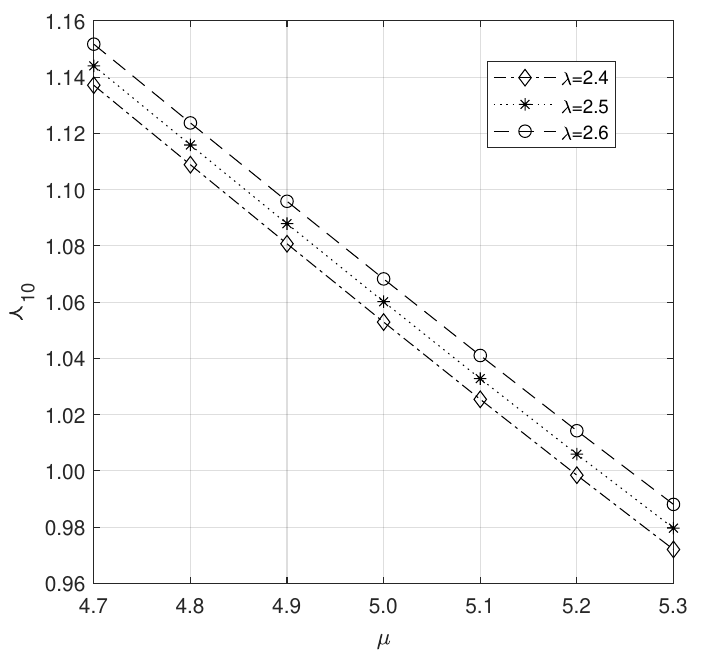} 
			\caption{Mean number of busy servers  vs $\lambda$ and $\mu$} 
			\label{fig:5}
			\end{minipage} 
		\end{adjustwidth}
	\end{figure}
	\begin{figure}
	\begin{adjustwidth}{-.5cm}{-.5cm}
		\begin{minipage}{0.55\textwidth}
			\includegraphics[width=1\linewidth]{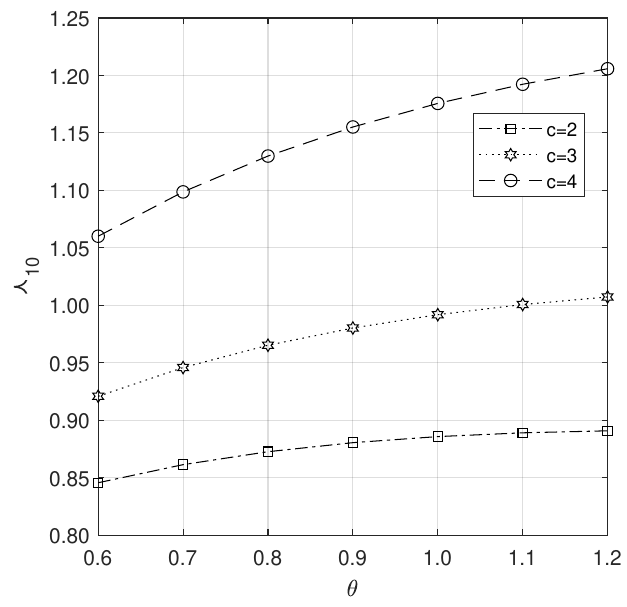} 
			\caption{Mean number of busy servers vs $c$ and $\theta$ } 
			\label{fig:6}
		\end{minipage}
		\begin{minipage}{0.55\textwidth}
			\includegraphics[width=1\linewidth]{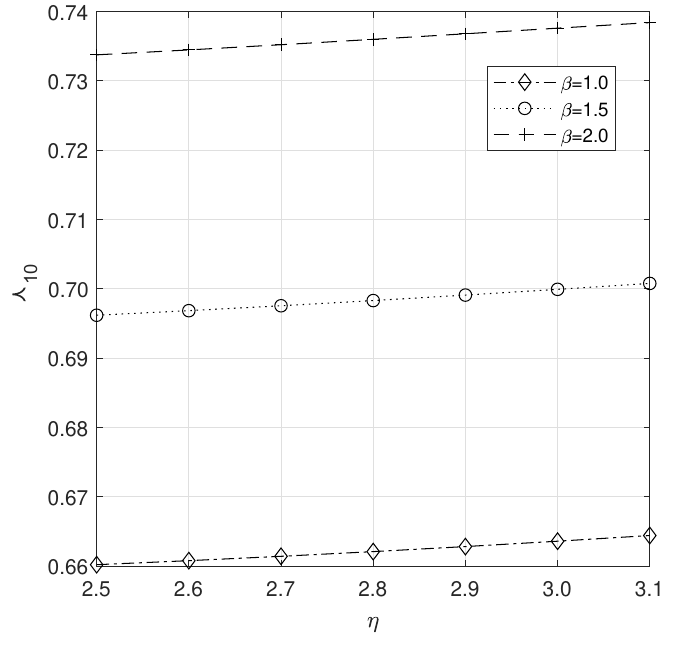} 
			\caption{Mean number of busy servers vs $\beta$ and $\eta$}
			\label{fig:7}
		\end{minipage}
	\end{adjustwidth}
\end{figure}
\end{ex}
\newpage
\vspace*{.5cm}
{\footnotesize
	\setlength\LTleft{-1.4cm}
\begin{longtable}{|c|c|c|c|cccc|ccc|ccc|}
		\caption{Expected customer lost rate with the variation of parameters}
		\label{tab2}\\
		\hline
		\multirow{2}{*}{$c$}  & \multirow{2}{*}{$\eta$}   & \multirow{2}{*}{$\theta$}   & \multirow{2}{*}{$\beta$} & \multicolumn{1}{c|}{$\lambda$} & \multicolumn{3}{c|}{2.4}                                            & \multicolumn{3}{c|}{2.5}                                                 & \multicolumn{3}{c|}{2.6}                                                 \\ \cline{5-14} 
		&                      &                      &                    & \multicolumn{1}{c|}{$\mu$} & \multicolumn{1}{c|}{4.9} & \multicolumn{1}{c|}{5}        & 5.1      & \multicolumn{1}{c|}{4.9}      & \multicolumn{1}{c|}{5}        & 5.1      & \multicolumn{1}{c|}{4.9}      & \multicolumn{1}{c|}{5}        & 5.1      \\ \hline
		\endfirsthead
		\multicolumn{14}{c}%
		{{\bfseries Table \thetable\ continued from previous page}} \\
		\hline
		\multirow{2}{*}{$ c $}  & \multirow{2}{*}{$ \eta $}   & \multirow{2}{*}{$ \theta $}   & \multirow{2}{*}{$ \beta $} & \multicolumn{1}{c|}{$ \lambda $} & \multicolumn{3}{c|}{2.4}                                            & \multicolumn{3}{c|}{2.5}                                                 & \multicolumn{3}{c|}{2.6}                                                 \\ \cline{5-14} 
		&                      &                      &                    & \multicolumn{1}{c|}{$ \mu $} & \multicolumn{1}{c|}{4.9} & \multicolumn{1}{c|}{5}        & 5.1      & \multicolumn{1}{c|}{4.9}      & \multicolumn{1}{c|}{5}        & 5.1      & \multicolumn{1}{c|}{4.9}      & \multicolumn{1}{c|}{5}        & 5.1      \\ \hline
		\endhead
		\multirow{27}{*}{2} & \multirow{9}{*}{2.6} & \multirow{3}{*}{0.6} & 1                  & \multicolumn{2}{c|}{0.335937}                     & \multicolumn{1}{c|}{0.335762} & 0.335567 & \multicolumn{1}{c|}{0.344799} & \multicolumn{1}{c|}{0.344599} & 0.344377 & \multicolumn{1}{c|}{0.353376} & \multicolumn{1}{c|}{0.353149} & 0.352898 \\ \cline{4-14} 
		&                      &                      & 1.5                & \multicolumn{2}{c|}{0.299024}                     & \multicolumn{1}{c|}{0.298380} & 0.297654 & \multicolumn{1}{c|}{0.306945} & \multicolumn{1}{c|}{0.306244} & 0.305457 & \multicolumn{1}{c|}{0.314594} & \multicolumn{1}{c|}{0.313834} & 0.312982 \\ \cline{4-14} 
		&                      &                      & 2                  & \multicolumn{2}{c|}{0.258399}                     & \multicolumn{1}{c|}{0.256290} & 0.253947 & \multicolumn{1}{c|}{0.265413} & \multicolumn{1}{c|}{0.263225} & 0.260797 & \multicolumn{1}{c|}{0.272160} & \multicolumn{1}{c|}{0.269886} & 0.267368 \\ \cline{3-14} 
		&                      & \multirow{3}{*}{0.7} & 1                  & \multicolumn{2}{c|}{0.363575}                     & \multicolumn{1}{c|}{0.363396} & 0.363198 & \multicolumn{1}{c|}{0.373551} & \multicolumn{1}{c|}{0.373347} & 0.373121 & \multicolumn{1}{c|}{0.383228} & \multicolumn{1}{c|}{0.382997} & 0.382740 \\ \cline{4-14} 
		&                      &                      & 1.5                & \multicolumn{2}{c|}{0.326191}                     & \multicolumn{1}{c|}{0.325532} & 0.324791 & \multicolumn{1}{c|}{0.335141} & \multicolumn{1}{c|}{0.334424} & 0.333620 & \multicolumn{1}{c|}{0.343806} & \multicolumn{1}{c|}{0.343027} & 0.342155 \\ \cline{4-14} 
		&                      &                      & 2                  & \multicolumn{2}{c|}{0.284384}                     & \multicolumn{1}{c|}{0.282202} & 0.279775 & \multicolumn{1}{c|}{0.292336} & \multicolumn{1}{c|}{0.290069} & 0.287551 & \multicolumn{1}{c|}{0.300004} & \multicolumn{1}{c|}{0.297646} & 0.295034 \\
\hline
	\end{longtable}
}
\newpage
\vspace*{1cm}
{\footnotesize
	\setlength\LTleft{-1.4cm}
	\begin{longtable}{|c|c|c|c|cccc|ccc|ccc|}
		\caption{Fraction of successful rate of retrial under the parameter variation}
		\label{tab3}\\
		\cline{1-14} 
		\multirow{2}{*}{$c$}  & \multirow{2}{*}{$\eta$}   & \multirow{2}{*}{$\theta$}   & \multirow{2}{*}{$\beta$} & \multicolumn{1}{c|}{$\lambda$} & \multicolumn{3}{c|}{2.4}                                            & \multicolumn{3}{c|}{2.5}                                                 & \multicolumn{3}{c|}{2.6}                                                 \\ \cline{5-14} 
		&                      &                      &                    & \multicolumn{1}{c|}{$\mu$} & \multicolumn{1}{c|}{4.9} & \multicolumn{1}{c|}{5}        & 5.1      & \multicolumn{1}{c|}{4.9}      & \multicolumn{1}{c|}{5}        & 5.1      & \multicolumn{1}{c|}{4.9}      & \multicolumn{1}{c|}{5}        & 5.1      \\ \hline
		\endfirsthead
		\multicolumn{14}{c}%
		{{\bfseries Table \thetable\ continued from previous page}} \\
		\cline{1-14} 
		\multirow{2}{*}{$ c $}  & \multirow{2}{*}{$ \eta $}   & \multirow{2}{*}{$ \theta $}   & \multirow{2}{*}{$ \beta $} & \multicolumn{1}{c|}{$ \lambda $} & \multicolumn{3}{c|}{2.4}                                            & \multicolumn{3}{c|}{2.5}                                                 & \multicolumn{3}{c|}{2.6}                                                 \\ \cline{5-14} 
		&                      &                      &                    & \multicolumn{1}{c|}{$ \mu $} & \multicolumn{1}{c|}{4.9} & \multicolumn{1}{c|}{5}        & 5.1      & \multicolumn{1}{c|}{4.9}      & \multicolumn{1}{c|}{5}        & 5.1      & \multicolumn{1}{c|}{4.9}      & \multicolumn{1}{c|}{5}        & 5.1      \\ \hline
		\endhead
		\multirow{15}{*}{2} & \multirow{9}{*}{2.6} & \multirow{3}{*}{0.6} & 1 & \multicolumn{2}{c|}{98.7165} & \multicolumn{1}{c|}{104.3344} & 110.0737 & \multicolumn{1}{c|}{80.4930} & \multicolumn{1}{c|}{84.9905} & 89.5778 & \multicolumn{1}{c|}{66.1961} & \multicolumn{1}{c|}{69.8320} & 73.5351 \\ \cline{4-14} 
		&  &  & 1.5 & \multicolumn{2}{c|}{127.8396} & \multicolumn{1}{c|}{136.9741} & 146.5567 & \multicolumn{1}{c|}{105.6185} & \multicolumn{1}{c|}{113.1000} & 120.9419 & \multicolumn{1}{c|}{87.9277} & \multicolumn{1}{c|}{94.1030} & 100.5702 \\ \cline{4-14} 
		&  &  & 2 & \multicolumn{2}{c|}{132.6868} & \multicolumn{1}{c|}{142.4772} & 152.7901 & \multicolumn{1}{c|}{110.0718} & \multicolumn{1}{c|}{118.1507} & 126.6569 & \multicolumn{1}{c|}{92.0150} & \multicolumn{1}{c|}{98.7337} & 105.8046 \\ \cline{3-14} 
		&  & \multirow{3}{*}{0.7} & 1 & \multicolumn{2}{c|}{98.6237} & \multicolumn{1}{c|}{104.2369} & 109.9712 & \multicolumn{1}{c|}{80.4062} & \multicolumn{1}{c|}{84.8996} & 89.4825 & \multicolumn{1}{c|}{66.1147} & \multicolumn{1}{c|}{69.7470} & 73.4462 \\ \cline{4-14} 
		&  &  & 1.5 & \multicolumn{2}{c|}{127.7995} & \multicolumn{1}{c|}{136.9321} & 146.5129 & \multicolumn{1}{c|}{105.5787} & \multicolumn{1}{c|}{113.0585} & 120.8986 & \multicolumn{1}{c|}{87.8885} & \multicolumn{1}{c|}{94.0621} & 100.5276 \\ \cline{4-14} 
		&  &  & 2 & \multicolumn{2}{c|}{132.6664} & \multicolumn{1}{c|}{142.4565} & 152.7690 & \multicolumn{1}{c|}{110.0509} & \multicolumn{1}{c|}{118.1294} & 126.6352 & \multicolumn{1}{c|}{91.9936} & \multicolumn{1}{c|}{98.7120} & 105.7825 \\ \cline{3-14} 
		&  & \multirow{3}{*}{0.8} & 1 & \multicolumn{2}{c|}{98.5413} & \multicolumn{1}{c|}{104.1504} & 109.8803 & \multicolumn{1}{c|}{80.3290} & \multicolumn{1}{c|}{84.8188} & 89.3979 & \multicolumn{1}{c|}{66.0422} & \multicolumn{1}{c|}{69.6713} & 73.3671 \\
\hline
	\end{longtable}	
}
\begin{ex}
	Table \ref{tab3} explains the fraction of the successful rate of retrial under the parameter variation.
	\begin{enumerate}
		\item As the rate $\lambda$ increases, the fraction of the successful rate of retrial decreases due to the overflow in the waiting hall.
		\item If $\mu$ increases, the number of customers in the waiting hall starts decreasing, and so a fraction of successful retrials increases. Similarly, the increase in the number of servers in the system causes an increase in the successful rate of retrials in the system.
		\item As we predicted, the rate $\theta$ reduces the fraction of successful retrials when it increases because more customers from the orbit enter the waiting hall.
		\item When $\beta$ increases, the inventory storage system maintains a positive stock level so that an arriving customer can purchase the item and leave the system without delay. Hence, it increases the fraction of the successful retrial rate.
		\item Suppose the server returns from vacation sooner than expected, then many customers get their service without a server delay. On the other hand, if the server is on vacation, a customer has to wait some time until his turn occurs. So, if the servers are already available in the system, his turn occurs quickly. Therefore, the vacation completion rate increases the fraction of the success rate of a retrial if it increases.
	\end{enumerate}
	\textbf{Insights:} When we approach a business economically, caring for the orbital customer's arrival will influence the growth of the business. So, this example explores the characteristics of the fraction of successful retrial rates of orbit customers with parameter variation. In such a way, the parameters $c$, $\beta$, and $\eta$ have great importance to overcome the loss of retrial customers.
\end{ex}
\begin{ex}
	This example examines the orbit and waiting hall customer's waiting times in Tables \ref{tab4} and \ref{tab5} respectively, using the parameters $\lambda, \mu, \theta, \beta, c$, and $\eta$. 
	\begin{enumerate}
		\item The parameters $\mu$ and $c$ have an equivalent impact on reducing the waiting time of a customer, whether they are in the queue or in the orbit. In other words, decreasing the average time it takes to complete a service or increasing the number of servers will result in a decrease in the number of customers waiting in the queue, which will also lead to a reduction in the number of customers in orbit. This ultimately leads to a reduction in the waiting time for both sets of customers.
		\item In a similar way, the replenishment rate helps the system avoid a shortage of inventory if it increases. Since the inventory is positive, the customers in the waiting hall do not need to wait a long time in the queue due to the lack of inventory. So, both customers' waiting times will be reduced.
		\item The vacation completion rate also contributes to reducing the waiting time of the customers in a queue or orbit. This rate ensures that the servers are mostly available when it increases.
		\item  If $\lambda$ increases, the number of customers in the waiting hall increases and soon overflows. Thus, both customers' waiting times in the queue and the orbit increase whenever it increases.
		
		\item Increasing $\theta$ only affects the waiting time of customers in the orbit, not those in the waiting hall. This is because the waiting hall quickly fills up, and an increase in $\theta$ results in a decrease in the number of customers in orbit, which ultimately leads to a reduction in their waiting time.
	\end{enumerate}
	\textbf{Insights:}
	It is crucial to analyze waiting time in customer-centric service systems in order to run a successful business. Ensuring customer satisfaction through the service system is paramount in this environment, as reducing customers' waiting time leads to higher levels of satisfaction and a better reputation for the company. This, in turn, leads to more customers returning to make additional purchases within the same system. Therefore, this example serves as a persuasive argument for readers and business people to prioritize reducing waiting times for their customers.
\end{ex}
{\footnotesize
	\setlength\LTleft{-1.4cm}
\begin{longtable}{|c|c|c|c|cccc|ccc|ccc|}
		\caption{Mean waiting time for orbit customers  under the impact of parameters}
		\label{tab4}\\
		\cline{1-14}
		\multirow{2}{*}{$c$}  & \multirow{2}{*}{$\eta$}   & \multirow{2}{*}{$\theta$}   & \multirow{2}{*}{$\beta$} & \multicolumn{1}{c|}{$\lambda$} & \multicolumn{3}{c|}{2.4}                                            & \multicolumn{3}{c|}{2.5}                                                 & \multicolumn{3}{c|}{2.6}                                                 \\ \cline{5-14} 
		&                      &                      &                    & \multicolumn{1}{c|}{$\mu$} & \multicolumn{1}{c|}{4.9} & \multicolumn{1}{c|}{5}        & 5.1      & \multicolumn{1}{c|}{4.9}      & \multicolumn{1}{c|}{5}        & 5.1      & \multicolumn{1}{c|}{4.9}      & \multicolumn{1}{c|}{5}        & 5.1      \\ \hline
		\endfirsthead
		\multicolumn{14}{c}%
		{{\bfseries Table \thetable\ continued from previous page}} \\
		\hline
		\multirow{2}{*}{$ c $}  & \multirow{2}{*}{$ \eta $}   & \multirow{2}{*}{$ \theta $}   & \multirow{2}{*}{$ \beta $} & \multicolumn{1}{c|}{$ \lambda $} & \multicolumn{3}{c|}{2.4}                                            & \multicolumn{3}{c|}{2.5}                                                 & \multicolumn{3}{c|}{2.6}                                                 \\ \cline{5-14} 
		&                      &                      &                    & \multicolumn{1}{c|}{$ \mu $} & \multicolumn{1}{c|}{4.9} & \multicolumn{1}{c|}{5}        & 5.1      & \multicolumn{1}{c|}{4.9}      & \multicolumn{1}{c|}{5}        & 5.1      & \multicolumn{1}{c|}{4.9}      & \multicolumn{1}{c|}{5}        & 5.1      \\ \hline
		\endhead
		\multirow{18}{*}{2} & \multirow{9}{*}{2.6} & \multirow{3}{*}{0.6} & 1 & \multicolumn{2}{c|}{0.033279} & \multicolumn{1}{c|}{0.031156} & 0.029165 & \multicolumn{1}{c|}{0.038288} & \multicolumn{1}{c|}{0.035851} & 0.033563 & \multicolumn{1}{c|}{0.043807} & \multicolumn{1}{c|}{0.041027} & 0.038417 \\ \cline{4-14} 
		&  &  & 1.5 & \multicolumn{2}{c|}{0.025736} & \multicolumn{1}{c|}{0.023661} & 0.021773 & \multicolumn{1}{c|}{0.029994} & \multicolumn{1}{c|}{0.027609} & 0.025436 & \multicolumn{1}{c|}{0.034780} & \multicolumn{1}{c|}{0.032052} & 0.029563 \\ \cline{4-14} 
		&  &  & 2 & \multicolumn{2}{c|}{0.022679} & \multicolumn{1}{c|}{0.021326} & 0.020101 & \multicolumn{1}{c|}{0.027110} & \multicolumn{1}{c|}{0.025514} & 0.024070 & \multicolumn{1}{c|}{0.032177} & \multicolumn{1}{c|}{0.030308} & 0.028615 \\ \cline{3-14} 
		&  & \multirow{3}{*}{0.7} & 1 & \multicolumn{2}{c|}{0.025422} & \multicolumn{1}{c|}{0.023791} & 0.022263 & \multicolumn{1}{c|}{0.029257} & \multicolumn{1}{c|}{0.027384} & 0.025628 & \multicolumn{1}{c|}{0.033485} & \multicolumn{1}{c|}{0.031347} & 0.029342 \\ \cline{4-14} 
		&  &  & 1.5 & \multicolumn{2}{c|}{0.019850} & \multicolumn{1}{c|}{0.018261} & 0.016815 & \multicolumn{1}{c|}{0.023129} & \multicolumn{1}{c|}{0.021302} & 0.019638 & \multicolumn{1}{c|}{0.026814} & \multicolumn{1}{c|}{0.024724} & 0.022818 \\ \cline{4-14} 
		&  &  & 2 & \multicolumn{2}{c|}{0.017888} & \multicolumn{1}{c|}{0.016829} & 0.017681 & \multicolumn{1}{c|}{0.021359} & \multicolumn{1}{c|}{0.020112} & 0.018982 & \multicolumn{1}{c|}{0.025324} & \multicolumn{1}{c|}{0.023865} & 0.022543 \\ \cline{3-14} 
		&  & \multirow{3}{*}{0.8} & 1 & \multicolumn{2}{c|}{0.020201} & \multicolumn{1}{c|}{0.018899} & 0.017681 & \multicolumn{1}{c|}{0.023255} & \multicolumn{1}{c|}{0.021759} & 0.020358 & \multicolumn{1}{c|}{0.026622} & \multicolumn{1}{c|}{0.024914} & 0.023314 \\
\end{longtable}
}
\vspace*{.5cm}
{\footnotesize
	\setlength\LTleft{-1.4cm}
	\begin{longtable}{|c|c|c|c|cccc|ccc|ccc|}
		\caption{Mean waiting time for a waiting hall customer under the impact of parameters different parameters}
		\label{tab5}\\
		\hline
		\multirow{2}{*}{$c$}  & \multirow{2}{*}{$\eta$}   & \multirow{2}{*}{$\theta$}   & \multirow{2}{*}{$\beta$} & \multicolumn{1}{c|}{$\lambda$} & \multicolumn{3}{c|}{2.4}                                            & \multicolumn{3}{c|}{2.5}                                                 & \multicolumn{3}{c|}{2.6}                                                 \\ \cline{5-14} 
		&                      &                      &                    & \multicolumn{1}{c|}{$\mu$} & \multicolumn{1}{c|}{4.9} & \multicolumn{1}{c|}{5}        & 5.1      & \multicolumn{1}{c|}{4.9}      & \multicolumn{1}{c|}{5}        & 5.1      & \multicolumn{1}{c|}{4.9}      & \multicolumn{1}{c|}{5}        & 5.1      \\ 	 \hline
		\endhead
		\multirow{27}{*}{2} & \multirow{9}{*}{2.6} & \multirow{3}{*}{0.6} & 1 & \multicolumn{2}{c|}{0.975550} & \multicolumn{1}{c|}{0.972322} & 0.969182 & \multicolumn{1}{c|}{1.017790} & \multicolumn{1}{c|}{1.014531} & 1.011356 & \multicolumn{1}{c|}{1.063851} & \multicolumn{1}{c|}{1.060560} & 1.057346 \\ \cline{4-14} 
		&  &  & 1.5 & \multicolumn{2}{c|}{0.918359} & \multicolumn{1}{c|}{0.913785} & 0.909720 & \multicolumn{1}{c|}{0.954241} & \multicolumn{1}{c|}{0.949683} & 0.945658 & \multicolumn{1}{c|}{0.993067} & \multicolumn{1}{c|}{0.988520} & 0.984533 \\ \cline{4-14} 
		&  &  & 2 & \multicolumn{2}{c|}{0.898471} & \multicolumn{1}{c|}{0.889304} & 0.880625 & \multicolumn{1}{c|}{0.931381} & \multicolumn{1}{c|}{0.921842} & 0.912809 & \multicolumn{1}{c|}{0.966996} & \multicolumn{1}{c|}{0.957039} & 0.947606 \\ \cline{3-14} 
		&  & \multirow{3}{*}{0.7} & 1 & \multicolumn{2}{c|}{1.031784} & \multicolumn{1}{c|}{1.028592} & 1.025490 & \multicolumn{1}{c|}{1.076556} & \multicolumn{1}{c|}{1.073328} & 1.070184 & \multicolumn{1}{c|}{1.125348} & \multicolumn{1}{c|}{1.122080} & 1.118891 \\ \cline{4-14} 
		&  &  & 1.5 & \multicolumn{2}{c|}{0.979389} & \multicolumn{1}{c|}{0.974836} & 0.970790 & \multicolumn{1}{c|}{1.018319} & \multicolumn{1}{c|}{1.013766} & 1.009745 & \multicolumn{1}{c|}{1.060450} & \multicolumn{1}{c|}{1.055891} & 1.051891 \\ \cline{4-14} 
		&  &  & 2 & \multicolumn{2}{c|}{0.961433} & \multicolumn{1}{c|}{0.952232} & 0.943521 & \multicolumn{1}{c|}{0.997523} & \multicolumn{1}{c|}{0.987934} & 0.978854 & \multicolumn{1}{c|}{1.036583} & \multicolumn{1}{c|}{1.026558} & 1.017060 \\ \cline{3-14} 
		&  & \multirow{3}{*}{0.8} & 1 & \multicolumn{2}{c|}{1.079977} & \multicolumn{1}{c|}{1.076826} & 1.073767 & \multicolumn{1}{c|}{1.126822} & \multicolumn{1}{c|}{1.123629} & 1.120524 & \multicolumn{1}{c|}{1.177840} & \multicolumn{1}{c|}{1.174602} & 1.171448 \\
\end{longtable}
}
\begin{ex}
	This example explores different performance measures for the parameters $N$ and $S$.
	\begin{enumerate}
		\item Figure \ref{fig:8} shows that increasing the size of the waiting hall would increase the $ETC$. Because $\curlywedge_6$ increases when $N$ increases. While varying the maximum inventory level $S$, a local minimum is obtained.
		\item By increasing the waiting hall size, $\curlywedge_6$ experiences an increase. Additionally, Figure \ref{fig:9} shows the optimum mean waiting time for customers in the waiting hall, which is determined by varying the maximum inventory level.
		\item Figure \ref{fig:10} shows that increasing the values of $S$ and $N$, decreases the mean waiting time of the customer in orbit. Because orbit customers can enter the waiting hall as its size increases. Similarly, when $S$ is increasing, customers in the waiting hall get service and leave quickly, which gives orbital customers a chance to enter the waiting hall.
		\item As the number of customers in the waiting hall and items in the inventory increase, the possibility that a server goes on vacation is reduced, which is clear from Figure \ref{fig:11}.
		\item From Figure \ref{fig:12}, we see that increasing the waiting hall size reduces customer loss. When $N=5, 10$ the expected customer loss ratio is small, whereas when we choose a large $N$, the expected customer loss ratio is high. Similarly, if the product availability increases, the customer in the waiting hall can get their service as much as possible and leaves the system soon. If a new customer faces that a waiting hall is full, his loss rate is to be controlled.
		
	\end{enumerate}
	\noindent\textbf{Insights:} When there is a large waiting room, the number of sales goes up because customers are less likely to leave, and the average time  customers wait in the orbit goes down. Also, the convexity obtained for total cost and waiting time for the waiting hall suggests that an optimal inventory can be maintained.
		\begin{figure}[]
		\begin{adjustwidth}{-.6cm}{-.6cm}
			\begin{minipage}[b]{0.55\linewidth}
				\includegraphics[width=1\linewidth]{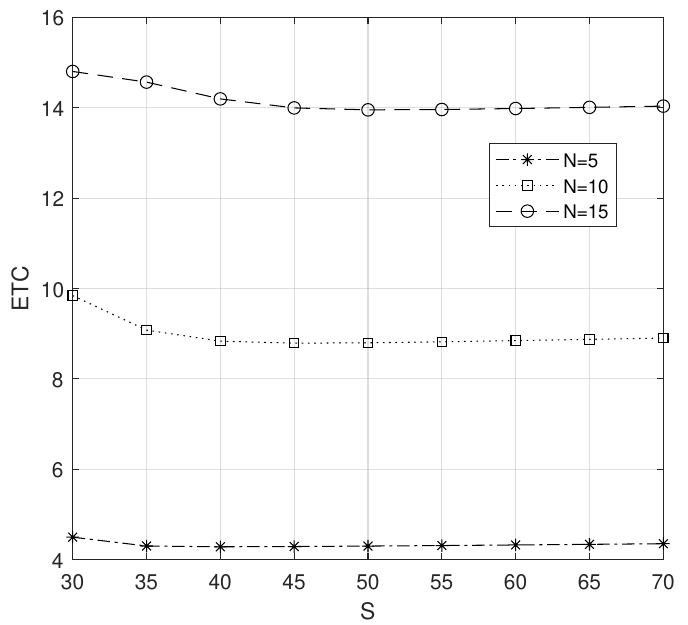} 
				\caption{Expected total cost vs $N$ and $S$ } 
				\label{fig:8}
			\end{minipage}
			\begin{minipage}[b]{0.56\linewidth}
				\includegraphics[width=1\linewidth]{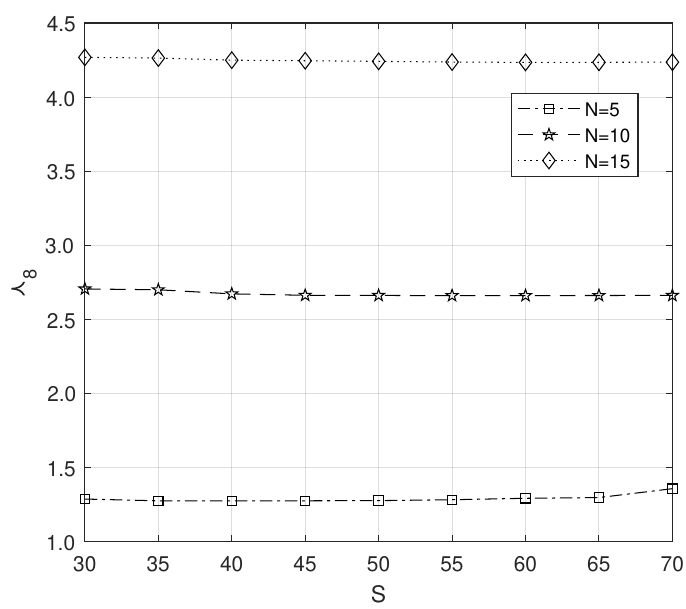} 
				\caption{Mean waiting time for waiting hall vs $N$ and $S$}
				\label{fig:9}
			\end{minipage} 
			\begin{minipage}[h]{0.56\linewidth}
				\includegraphics[width=1\linewidth]{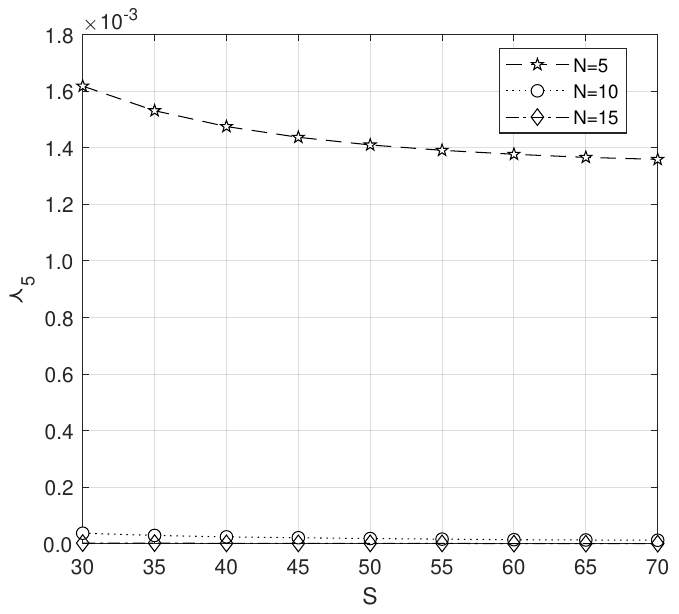} 
				\caption{Mean waiting time for Orbit vs $N$ and $S$ } 
				\label{fig:10}
			\end{minipage} 
			\begin{minipage}[h]{0.55\linewidth}
				\includegraphics[width=1\linewidth]{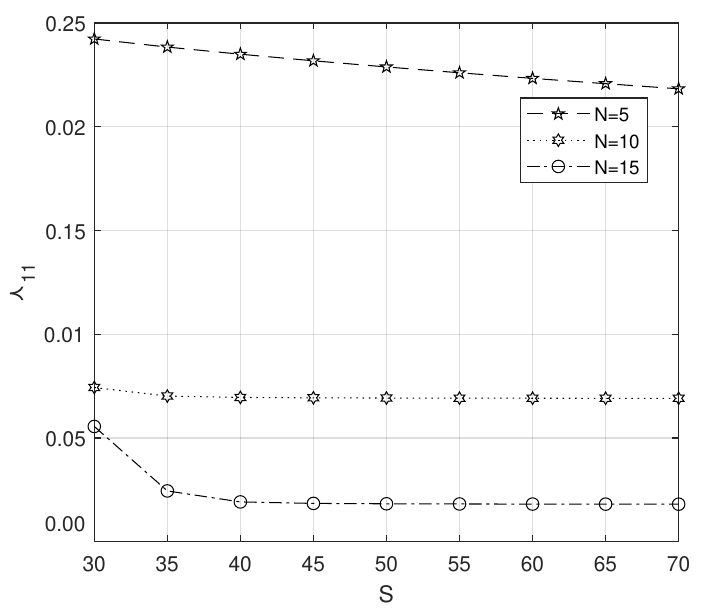} 
				\caption{Mean number of server is on vacation vs $N$ and $S$}
				\label{fig:11}
			\end{minipage} 
		\end{adjustwidth}
	\end{figure}
	\begin{figure}
		\includegraphics{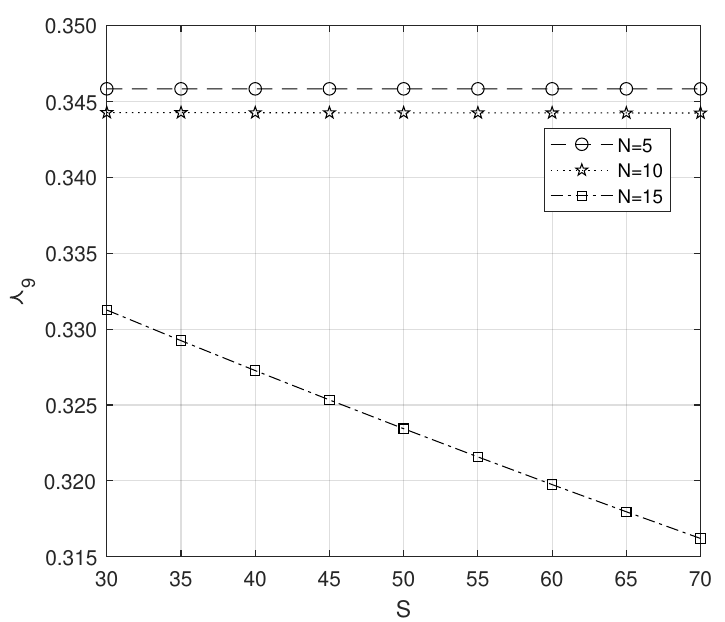}
		\caption{Mean number of customer lost vs $N$ and $S$}
		\label{fig:12}
	\end{figure}
\end{ex}
\section{Conclusion}\label{con}
\indent \indent 	
In this article, we investigated a multi-server queueing-inventory system with asynchronous multiple-server vacations. In addition, this model explored the retrial facility and $(s, Q)$ reordering policy for the replenishment process. As per the assumptions, a six-dimensional stochastic process is framed and transformed into an infinitesimal matrix by the transition of states. The stability condition and joint probability vector are obtained through the matrix geometric approximation method, based on the structure of the infinitesimal generator matrix. Moreover, the expected total cost function is established by considering certain system performance measures.

From the analysis of this paper, we observe the following results:
\subsection{Observations}
These observations are obtained by the numerical investigations:
\begin{enumerate}
	\item From the parameter analysis, $\beta$, $\eta$, and $\mu$ are directly proportional to an expected total cost, whereas $\lambda$, $\theta$ and $c$ are inversely proportional to the expected total cost.
	\item The parameters $\lambda$, $\theta$ were calculated using the mean number of busy servers with a direct variation property. On the other hand, $\beta$, $\eta$, and $\mu$ contribute to give a rest period to the servers and react with an indirect variation.
	\item The expected loss rate of a customer is controlled by the increase of the parameters $\beta$, $\eta$, $c$ and $\mu$.
	\item The fraction of the successful retrial rate of an orbit customer is slowed down due to the overflow of the waiting hall. To control such an overflow, one can increase the values of parameters $\beta$, $\eta$, $c$, and $\mu$ as required.
	\item The waiting time of queue and orbit customers has an indirect variation relation with the parameters $\beta$, $\eta$, $c$ and $\mu$. As the situation requires, one can vary the parameter value to get the desired output.
\end{enumerate}
\subsection{Limitations}
When we consider $c=1$, this model converts into a single server QIS with multiple server vacations. If the retrial customer is independent of the customer in the orbit, then the corresponding classical retrial policy is converted to a constant retrial policy. Assuming $p=1$, then the proposed system sends the overflowing customers to an orbit compulsorily whereas if we assume $p=0$ all the customers are considered to be lost.
\subsection{Future Direction}
The proposed model may be discussed with synchronous vacations among the multi-server queueing-inventory system in future research.

\end{document}